\documentclass[10pt,reqno]{article}
\usepackage{amsthm}
\usepackage[numbers]{natbib}
\usepackage{amsmath}
\usepackage{amsmath}
\usepackage{bm}
\usepackage{float}
\usepackage{abstract}
\usepackage{amssymb}
\usepackage{setspace}
\usepackage{graphicx}
\usepackage{multirow}
\usepackage{booktabs}
\usepackage{multirow}
\usepackage{geometry}
\usepackage{xcolor}
\usepackage{caption}
\usepackage{subcaption}
\usepackage{hyperref}
\hypersetup{
	colorlinks=true,
	linkcolor=blue,
	filecolor=blue,
	urlcolor=red,
	citecolor=cyan
}

\numberwithin{equation}{section}
\newtheorem{theorem}{Theorem}[section]

\newtheorem{lemma}{Lemma}[section]

\theoremstyle{definition}
\newtheorem{definition}{Definition}[section]

\newtheorem{corollary}{Corollary}[section]

\theoremstyle{definition}
\newtheorem{remark}{Remark}[section]

\newcommand{\beq}{\begin{equation}}
\newcommand{\eeq}{\end{equation}}

\newcommand{\GL}{\Lambda}
\newcommand{\eqnref}[1]{(\ref {#1})}

\newcommand{\Ge}{\epsilon}
\newcommand{\Gs}{\sigma}

\title{Existence and uniqueness of Generalized Polarization Tensors vanishing structures}
\author{
{Fanbo Sun}
\thanks{School of Mathematics and Statistics, Central South University, Changsha, 410083, Hunan Province, China. \ \ Email: fanbo\_sun2023@163.com}\quad\quad
{Youjun Deng}
\thanks{Corresponding author. School of Mathematics and Statistics, Central South University, Changsha, 410083, Hunan Province, China. \ \ Email: youjundeng@csu.edu.cn; dengyijun\_001@163.com}
}
\date{}
\begin{document}
	\maketitle

	\begin{abstract}
		This paper is concerned with the open problem proposed in Ammari et. al. Commun. Math.
		Phys, 2013. We first investigate the existence and uniqueness of Generalized Polarization Tensors (GPTs) vanishing structures locally in both two and three dimension by fixed point theorem. Employing the Brouwer Degree Theory and the local uniqueness, we prove that for any radius configuration of $N+1$ layers concentric disks (balls) and a fixed core conductivity, there exists at least one piecewise homogeneous conductivity distribution which achieves the $N$-GPTs vanishing. Furthermore, we establish a global uniqueness result for the case of proportional radius settings, and derive an interesting asymptotic configuration for structure with thin coatings. Finally, we present some numerical examples to validate our theoretical conclusions.
		\medskip
		
		
		
	\end{abstract}
	\section{Introduction}
	
	\quad\  Consider the conductivity problem
	\begin{equation}\label{conductivity problem}
		\begin{cases}
			\begin{aligned}
				&\nabla \cdot ((\sigma_{\Omega}\chi_{\Omega}+\chi_{\Omega_0})\nabla u)=0 &  \text{in } \mathbb{R}^d,\\
				&(u-H)(x) = O(|x|^{1-d}) &  \text{as } |x|\to \infty.
			\end{aligned}
		\end{cases}
	\end{equation}
	where $d=2,3$ and $\Omega$ is an inclusion inserted into infinite homogeneous background $\Omega_0=\mathbb{R}^d\backslash \Omega$, $\chi$ presents the characteristic function and $H(x)$ is a harmonic function in $\mathbb{R}^d$. The conductivity $\sigma_{\Omega}$ of inclusion is different from the background, which cause some perturbation to the background fields. In the inverse conductivity problem, the measurement of boundary data is utilized to reconstruct unknown shapes and conductivity distribution of inclusions embedded in background. However, there are certain types of inclusions that only perturb the background fields very slightly (even not at all). These inclusions are referred to near (completely) cloaking structures, as they render themselves invisible to probing by electrical impedance tomography (EIT) \cite{you2020combined,ji2021neutral,blaasten2020recovering}.
	
	The cloaking by transmission optics is used to construct a singular conductivity distribution whose DtN map is exactly the same as constant conductivity distribution \cite{greenleaf}. However, this
	transformation induces the singularity of material. To overcome, the near cloaking structures developed by Kohn \textit{et al}. \cite{kohn2008cloaking} are widely applied in practice. They use a regular transformation to push forward the material constant into a small ball (with radius $\rho$) instead of a singular point, and the near cloaking effects can be estimated to be of order $\rho^d$. We also refer to \cite{DLU171,DLU172,LLRU15} for near cloaking anisotropic structures. In addition, neutral inclusions which do not cause any perturbation to uniform background fields have been extensively studied and widely used in designing invisibility cloaking structures with metamaterials in various contexts \cite{kang2019construction,kang2021polarization,kang2022existence,lim2020inclusions,nguyen2016cloaking}.
	
	The generalized polarization tensors (GPTs) vanishing structures were first put forward by Ammari \textit{et al}. \cite{ammari2013enhancement1}, where they designed the structure by multi-coated concentric disks or balls and proved the near cloaking effects can be enhanced to the order of $\rho^{d+2N}$ by $N$-GPTs vanishing structure.
The GPTs, which are extension of the polarization tensor introduced by Schiffer and Szego{\"a} \cite{schiffer1949virtual} to higher order. It's well known that GPTs are defined as an asymptotic sense and they carry geometric information about the inclusion. Indeed, GPTs have been widely used in reconstructing small inclusions
	and shape description \cite{ammari2014reconstruction2,bruhl2003direct,deng2024identifying}. The GPTs vanishing structure is also concerned as neutral inclusions in asymptotic sense. In addition, GPTs-vanishing structures of general shape have also been proposed \cite{feng2017construction,kang2022existence}. However, as an important problem reported initially, the existence of high order GPTs vanishing structures with multi-coated concentric disks, has seen little significant progress recently.
	
	The purpose of this paper is to prove that if the core $\Omega_{N+1}$ has any fixed constant conductivity, then there exist $N$ coatings surrounding $\Omega_{N+1}$ such that the inclusion $\Omega$ becomes an $N$-GPTs vanishing structure. This result is independent of radius settings, that is, for any $N+1$-layer concentric disks (balls) with given radius, there exists a suitable conductivity distribution to achieve the $N$-GPTs vanishing. To address this problem, We first show that the $N$-GPTs vanishing structure exists and is unique under some smallness assumptions on the structure. Then we derive the continuous differentiability of GPTs respect to the conductivity contrast $\eta$. This derivation relies on the integral equation representation with layer potential technique. Afterwards, by delicate and detailed analysis, we find $N$-GPTs vanishing structures employing the homotopic invariance of Brouwer Degree (see Theorem \ref{main result}). We analyze the mapping properties of GPTs respect to $\eta$ and compute the value of Brouwer Degree, which together with the locally uniqueness result, we prove the main results. Finally, we study a type of structures with proportional radius settings and derive a uniqueness result. Besides, we notice that the GPTs vanishing structures designed by extreme thin coats require high contrast conductivity setting, and the contrast between the outermost layer and background should be weaker than others. The numerical experiments are in collaborate with our theoretical findings.
	
	The organization of this paper is as follows. In the following section, we introduce the layer potential technique and the GPTs. Our main results are presented in Section 3. In section 4, we show the existence and uniqueness of $N$-GPTs vanishing structures locally. In Section \ref{sec:3}, we first prove the continuously differentiable of GPTs respect to $\eta$. Then we derive the existence of the $N$-GPTs vanishing structures for any fixed core conductivity using Brouwer Degree Theory. Section \ref{sec:proportional}  focuses on deriving the uniqueness of $N$-GPTs vanishing structures under proportional radius settings, the extreme case is also discussed in this part. In Section \ref{sec:5} we present some numerical experiments to verify our results.

	\section{Layer potentials and GPTs}
  \quad \ \ In this section, we shall introduce the GPTs and related integral system by layer potential technique. Let $\Omega\subset \mathbb{R}^d$ be a piecewise homogeneous domain divided by $N$ parts, i.e. $\Omega_k$, which are surrounded by $C^{1,\alpha}$($0<\alpha<1$) closed and nonintersecting surfaces $\Gamma_k$, $k=1,2,\dots,N$, and $\Omega_0=\mathbb{R}^d\backslash \Omega$ be the background medium. Each region $\Omega_k$ is filled with homogeneous medium, that is
	\begin{equation}
		\sigma(x) = \sigma_k, \quad x\in \Omega_k,\ k=0,1,\dots,N.
	\end{equation}
It is noted that the solution $u$ of \eqref{conductivity problem} satisfies the following transmission conditions:
	\begin{equation}\label{transmission}
		u|_{+}=u|_{-}, \quad \sigma_{k-1}\partial_{\nu_k}|_{+}=\sigma_{k}\partial_{\nu_k}|_{-} \quad \text{on } \Gamma_k,
	\end{equation}
	where the notation $\nu_k$ indicate the outward normal on $\Gamma_k$.
	
	Let $G(x)$ be the fundamental solution to
	the Laplace equation in $\mathbb{R}^d$, that is
	\begin{equation}\label{fundamental}
		G(x)=\begin{cases}
			\frac{1}{2\pi}\ln|x|,  &\text{  } d=2,\\
			\frac{1}{(2-d)\omega_d}|x|^{2-d}, &\text{  } d=3,
		\end{cases}
	\end{equation}
	where $\omega_d$ is the area of the unit sphere in $\mathbb{R}^d$. For a bounded simple closed $C^{1,\alpha}$ surface $\Gamma$, the single and double layer potentials with density $\phi\in L^2(\Gamma)$ are defined by
	
	\begin{equation}\label{single}
		\mathcal{S}_{\Gamma}[\phi]({x}):=\int_{\Gamma}G({x-y})\phi(y)d\sigma_{{y}},\quad x\in \mathbb{R}^d,\notag
	\end{equation}
	\begin{equation}\label{double}
		\mathcal{D}_{\Gamma}[\phi]({x}):=\int_{\Gamma}\frac{\partial}{\partial_{\nu_{y}}}G({x-y})\phi(y)d\sigma_{{y}},\quad x\in \mathbb{R}^d\backslash\Gamma,\notag
	\end{equation}
	where $\partial_{\nu_{y}}$ denotes the outward normal derivative respect to ${y}$-variables. The following jump relations hold
	\begin{equation}\label{sjump}
		\partial_{\nu}\mathcal{S}_{\Gamma}[\phi]({x})|_{\pm}=(\pm\frac{I}{2}+\mathcal{K}^{*}_{\Gamma})[\phi]\quad \text{on } \Gamma,
	\end{equation}
	\begin{equation}
		\mathcal{D}_{\Gamma}[\phi]({x})|_{\pm}=(\mp\frac{I}{2}+\mathcal{K}_{\Gamma})[\phi]\quad \text{on } \Gamma,
	\end{equation}
	where the Neumann-Poincar\'e operator $\mathcal{K}^{*}_{\Gamma}$ is defined by
	
	\begin{equation}\label{NP}
		\mathcal{K}^{*}_{\Gamma}[\phi](x):=\int_{\Gamma}\frac{\partial G({x-y})}{\partial_{\nu_{x}}}\phi(y)d\sigma_{{y}},
	\end{equation}
	and $\mathcal{K}_{\Gamma}$ is the $L^2$ adjoined of $\mathcal{K}^{*}_{\Gamma}$
	\begin{equation}\label{NPad}
		\mathcal{K}_{\Gamma}[\phi](x):=\int_{\Gamma}\frac{\partial G({x-y})}{\partial_{\nu_{y}}}\phi(y)d\sigma_{{y}}.
	\end{equation}
	By using layer potential technique, the solution $u$ to \eqref{conductivity problem} can be represented by
	\begin{equation}\label{solution}
		u(x)=H(x)+\sum_{k=1}^{N}\mathcal{S}_{\Gamma_k}[\phi_k](x),
	\end{equation}
	for some functions $\phi_k\in L_0^2(\Gamma_k)$, where the function space is defined by
	\begin{equation}
		L_0^2(\Gamma_k):=\left\{ \phi\in L^2(\Gamma_k):\int_{\Gamma_k}\phi=0\right\},\notag
	\end{equation}
	using \eqref{sjump}, the transmission conditions \eqref{transmission} can be rewritten by
	\begin{equation}\label{integral}
			\begin{bmatrix}
			\lambda_1I-\mathcal{K}^{*}_{\Gamma_1} & -\mathcal{K}_{\Gamma_2,\Gamma_1} &\cdots & -\mathcal{K}_{\Gamma_N,\Gamma_1}\\
			-\mathcal{K}_{\Gamma_1,\Gamma_2} & \lambda_2I-\mathcal{K}^{*}_{\Gamma_2} &\cdots & -\mathcal{K}_{\Gamma_N,\Gamma_2}\\
			\vdots & \vdots  & \ddots & \vdots\\
			-\mathcal{K}_{\Gamma_1,\Gamma_N} & -\mathcal{K}_{\Gamma_2,\Gamma_N}  &\cdots & \lambda_NI-\mathcal{K}^{*}_{\Gamma_N}
		\end{bmatrix}\begin{bmatrix}
			\phi_1\\\phi_2\\ \vdots \\\phi_N
		\end{bmatrix}=\begin{bmatrix}
			\partial_{\nu_1}H\\\partial_{\nu_2}H\\ \vdots \\\partial_{\nu_N}H
		\end{bmatrix},
	\end{equation}
	where  $\mathcal{K}_{\Gamma_{k_1},\Gamma_{k_2}}:L_0^2(\Gamma_{k_1})\to L_0^2(\Gamma_{k_2})$ denotes the normal derivative of single layer potential $\partial_{\nu_{k_2}}\mathcal{S}_{\Gamma_{k_1}}$ constrained on $\Gamma_{k_2}$ and
	\begin{equation}
		\lambda_k=\frac{\sigma_k+\sigma_{k-1}}{2(\sigma_k-\sigma_{k-1})},\quad k=1,2,\dots,N.
	\end{equation}
	
	Denote the integral operator on the left-hand side of \eqref{integral} by $\mathbb{J}_{\Omega}$. The solvability and uniqueness of above integral system can be established by the following theorem.
	\begin{theorem}\label{invertible}
		For any $\lambda_k\in (-\infty,-1/2]\cup[1/2,+\infty)$, $k=1,2,\dots,N$, the operator $\mathbb{J}_{\Omega}$ is invertible on $L_0^2(\Gamma_1)\times L_0^2(\Gamma_2)\times \dots\times L_0^2(\Gamma_N) \to L_0^2(\Gamma_1)\times L_0^2(\Gamma_2)\times \dots\times L_0^2(\Gamma_N)$.
		\begin{proof}
			For notation simplicity, we take function space $X=L_0^2(\Gamma_1)\times L_0^2(\Gamma_2)\times ...\times L_0^2(\Gamma_N)$, it follows from \eqref{integral} that $\mathbb{J}_{\Omega}$ can be decomposed by
			\begin{equation}
				\begin{aligned}\mathbb{J}_{\Omega}&=\begin{bmatrix}
					\lambda_1I-\mathcal{K}^{*}_{\Gamma_1} & 0 &\cdots & 0\\
					0 & \lambda_2I-\mathcal{K}^{*}_{\Gamma_2} &\cdots & 0\\
					\vdots & \vdots  & \ddots & \vdots\\
					0 & 0  &\cdots & \lambda_NI-\mathcal{K}^{*}_{\Gamma_N}
				\end{bmatrix}+\begin{bmatrix}
				0 & -\mathcal{K}_{\Gamma_2,\Gamma_1} &\cdots & -\mathcal{K}_{\Gamma_N,\Gamma_1}\\
				-\mathcal{K}_{\Gamma_1,\Gamma_2} & 0 &\cdots & -\mathcal{K}_{\Gamma_N,\Gamma_2}\\
				\vdots & \vdots  & \ddots & \vdots\\
				-\mathcal{K}_{\Gamma_1,\Gamma_N} & -\mathcal{K}_{\Gamma_2,\Gamma_N}  &\cdots & 0
			\end{bmatrix}\\
			&=:\mathbb{J}_{\Omega}^{(1)}+\mathbb{J}_{\Omega}^{(2)}.\notag
			\end{aligned}
			\end{equation}
		one can readily see $\mathbb{J}_{\Omega}^{(1)}$ is invertible on $X$ with $|\lambda_k|\geq 1/2$. Since the integral kernel of $\mathcal{K}_{\Gamma_{k_1},\Gamma_{k_2}}$ is analytic, the operator $\mathbb{J}_{\Omega}^{(2)}$ is compact. Thanks for Fredholm alternative, it suffices to prove the injectivity of $\mathbb{J}_{\Omega}$. Indeed, we assume that
		\begin{equation}
			\mathbb{J}_{\Omega}(\phi_1,\phi_2,...,\phi_N)^T=(0,0,...,0)^T,\notag
		\end{equation}
		where the superscript $T$ denotes the matrix transpose. Let
		\begin{equation}
			u(x):=\sum_{k=1}^{N}\mathcal{S}_{\Gamma_k}[\phi_k](x),\quad x\in\mathbb{R}^d.\notag
		\end{equation}
		It's easy to verify that $u$ satisfies the following  transmission problem
		\begin{equation}\label{gtransmission}
			\begin{cases}
					\Delta u=0,& \text{in }\mathbb{R}^d\backslash\cup_{k=1}^N\Gamma_k,\\
					u|_{+}=u|_{-}, & \text{on } \Gamma_k,\\
					(\lambda_k-1/2)\partial_{\nu}u|_{+}=(\lambda_k+1/2)\partial_{\nu}u|_{-}, & \text{on } \Gamma_k,\\
					u(x)=\mathcal{O}(|x|^{1-d}) &|x|\to \infty,
			\end{cases}
		\end{equation}
		we point out \eqref{gtransmission} has only trivial solution. Let $|\lambda_k|\neq 1/2$ for each $k$, the system reduces to a standard transmission problem, thus $u(x)=0$. Next, we suppose $|\lambda_k|=1/2$ for some $k$. If $\lambda_k=-1/2$, it immediately follows $\partial_{\nu}u|_{+}=0$ on $\Gamma_k$, hence $u(x)=0$ in $\mathbb{R}^d\backslash\cup_{i=1}^k\Omega_i$. Using the continuity condition $u|_{+}=u|_{-}$ on $\Gamma_k$ we get $u(x)=0$.
		
		If $\lambda_k=1/2$ for some $k$, that is $\partial_{\nu}u|_{-}=0$ on $\Gamma_k$, thus $u(x)=0,x\in \cup_{i=1}^k\Omega_i$ and consequently $u(x)=0$ in $\mathbb{R}^d$ from \eqref{transmission}. The jump formula \eqref{sjump} yields that
		\begin{equation}
			\phi_l=\frac{\partial\mathcal{S}_{\Gamma_l}}{\partial_{\nu}}|_{+}-\frac{\partial\mathcal{S}_{\Gamma_l}}{\partial_{\nu}}|_{-}=0.\notag
		\end{equation}
		This completes the proof.
		\end{proof}
	\end{theorem}
	
	With the uniqueness of $(\phi_1,\phi_2,\dots,\phi_N)$ determined by \eqref{integral},  we proceed to introduce the polarization tensors of multi-layer structures. For any multi-index $\alpha=(\alpha_1,\alpha_2,\dots,\alpha_d)$, let $x^{\alpha}=x_1^{\alpha_1}\dots x_d^{\alpha_d}$ and $\partial^{\alpha}=\partial_1^{\alpha_1}\dots\partial_d^{\alpha_d}$, in which $\partial_i=\partial/\partial x_i$. With the help of Taylor expansion for $y$ lies on a compact set we have
	\begin{equation}
		G(x-y)=\sum_{|\alpha|=0}^{+\infty}\frac{(-1)^{|\alpha|}}{\alpha!}\partial^{\alpha}G(x)y^{\alpha},\quad |x|\to\infty,\notag
	\end{equation}
	together with \eqref{solution}, we can obtain the following far-field expansion
	\begin{equation}\label{repre}
		\begin{aligned}
			(u-H)(x)&=\sum_{k=1}^N\int_{\Gamma_k}G(x-y)\phi_k(y)d\sigma_y\\
			&=\sum_{k=1}^N\sum_{|\alpha|=0}^{+\infty}\sum_{|\beta|=0}^{+\infty}\frac{(-1)^{|\alpha|}}{\alpha!\beta!}\partial^{\alpha}G(x)\partial^{\beta}H(0)\int_{\Gamma_k}y^{\alpha}\phi_{k,{\beta}}d\sigma_y,
		\end{aligned}
	\end{equation}
	as $|x|\to \infty$, where $(\phi_{1,\beta},\phi_{2,\beta},...,\phi_{N,\beta})=\mathbb{J}_{\Omega}^{-1}(\partial_{\nu_1}y^{\beta},\partial_{\nu_2}y^{\beta},...,\partial_{\nu_N}y^{\beta})^T$.
	\begin{definition}
		For $\alpha,\beta\in \mathbb{N}^d$, let $\phi_{k,\beta}, k=1,2,\dots,N$ be the solution of
		\begin{equation}
			\mathbb{J}_{\Omega}(\phi_{1,\beta},\dots,\phi_{N,\beta})^T=(\partial_{\nu_1}y^{\beta},\partial_{\nu_2}y^{\beta},\dots,\partial_{\nu_N}y^{\beta})^T. \notag
		\end{equation}
		Then the GPT $M_{\alpha\beta}$ is defined to be
		\begin{equation}\label{GPT}
			M_{\alpha\beta}:=\sum_{k=1}^N\int_{\Gamma_k}y^{\alpha}\phi_{k,{\beta}}d\sigma_y.
		\end{equation}
		Specially, denote $M_{\alpha\beta}$ by $M_{ij}$ with $|\alpha|=|\beta|=1$, the matrix $\mathbf{M}:=(M_{ij})_{i,j=1}^d$ is called P\'olya–Szego{\"a} polarization tensor.
	\end{definition}
	Through \eqref{repre}, we can get complete information about the far-field expansion of perturbed electric potential by GPTs
	\begin{equation}\label{expansion}
		(u-H)(x)=\sum_{|\alpha|=0}^{+\infty}\sum_{|\beta|=0}^{+\infty}\frac{(-1)^{|\alpha|}}{\alpha!\beta!}\partial^{\alpha}G(x)M_{\alpha\beta}\partial^{\beta}H(0),\quad \text{as }|x|\to \infty.
	\end{equation}

\section{Main results}
\quad \ \ In this section, we shall present the main results for existence and uniqueness of GPTs vanishing structures. We shall first introduce the contracted Generalized Polarization Tensors (CGPTs) and the concepts of $N$-GPTs vanishing structure in two and three dimension, respectively. The proof of main results shall be shown in sections afterwards.
	\subsection{Two dimensional case}
	\quad\ For any harmonic function $H(x)$ in $\mathbb{R}^2$, which admits the following expression
	\begin{equation}\label{H}
		H(x)=H(0)+\sum_{n=1}^{+\infty}r^n(a^c_n\cos n\theta+a^s_n\sin n\theta),
	\end{equation}
	where $x=(r\cos\theta,r\sin\theta)$. For multi-indices $\alpha,\beta$ with $|\alpha|=|\beta|=n$, we define $(a_{\alpha}^c)$ and $(a_{\beta}^s)$ with
	\begin{equation}
		\sum_{|\alpha|=n}a_{\alpha}^cx^{\alpha}=r^n\cos n\theta\quad \text{and} \quad \sum_{|\beta|=n}a_{\beta}^sx^{\alpha}=r^n\sin n\theta, \notag
	\end{equation}
	thus we define the contracted Generalized Polarization Tensors (CGPTs) by
	\begin{equation}\label{CGPT1}
		M_{mn}^{cc}:=\sum_{|\alpha|=m}^{+\infty}\sum_{|\beta|=n}^{+\infty}a_{\alpha}^ca_{\beta}^cM_{\alpha\beta},
	\end{equation}
	\begin{equation}\label{CGPT2}
		M_{mn}^{cs}:=\sum_{|\alpha|=m}^{+\infty}\sum_{|\beta|=n}^{+\infty}a_{\alpha}^ca_{\beta}^sM_{\alpha\beta},
	\end{equation}
	\begin{equation}\label{CGPT3}
		M_{mn}^{sc}:=\sum_{|\alpha|=m}^{+\infty}\sum_{|\beta|=n}^{+\infty}a_{\alpha}^sa_{\beta}^cM_{\alpha\beta},
	\end{equation}
	\begin{equation}\label{CGPT4}
		M_{mn}^{ss}:=\sum_{|\alpha|=m}^{+\infty}\sum_{|\beta|=n}^{+\infty}a_{\alpha}^sa_{\beta}^sM_{\alpha\beta},
	\end{equation}
	note the expansion of $G(x-y)$ as $|x|\to \infty$, that is
	\begin{equation}\label{G}
		G(x-y)=\sum_{n=1}^{+\infty}\frac{-1}{2\pi n}\Big\{ \frac{\cos n\theta_x}{r_x^n}r_y^n\cos n\theta_y+\frac{\sin n\theta_x}{r_x^n}r_y^n\sin n\theta_y \Big\}+C,
	\end{equation}
	where $x=(r_x\cos\theta_x,r_x\sin\theta_x)$, $y=(r_y\cos\theta_y,r_y\sin\theta_y)$ and $y\in \Gamma_k$, $k=1,\dots,N$. Plugging \eqref{H} and \eqref{G} into \eqref{expansion}, we get
	\begin{equation}\label{far-field}\begin{aligned}
			(u-H)(x)=&-\sum_{m=1}^{+\infty}\frac{\cos m\theta}{2\pi mr^m}\sum_{n=1}^{\infty}(M_{mn}^{cc}a_n^c+M_{mn}^{cs}a_n^s)\\
			&-\sum_{m=1}^{+\infty}\frac{\sin m\theta}{2\pi mr^m}\sum_{n=1}^{\infty}(M_{mn}^{sc}a_n^c+M_{mn}^{ss}a_n^s),
		\end{aligned}
	\end{equation}
	uniformly as $|x|\to +\infty$.
	The structure $\Omega$ with conductivity distribution $\sigma_{\Omega}$ is called $N$-GPTs vanishing structure which satisfies $M_{mn}^{cc}=M_{mn}^{cs}=M_{mn}^{sc}=M_{mn}^{ss}=0$ for any $m,n\leq N$.
	
	To obtain $N$-GPTs vanishing structures, Ammari \textit{et al}. \cite{ammari2013enhancement1} introduced multiple radial symmetric coatings. We shall prove that for any $N$,  there exists an $N+1$-layer radial symmetric structure can be designed to achieve the $N$-GPTs vanishing. Suppose $0<r_{N+1}<r_{N}<\dots<r_{2}<r_{1}$ and denote the coating domain by
	\begin{equation}
		\Omega_k:={r_{k+1}<r\leq r_{k}}, \quad k=1,2,\dots,N,\notag
	\end{equation}
	the core $\Omega_{N+1}=\{r\leq r_{N+1}\}$ coated by $\Omega_k$ and background media $\Omega_{0}=\{r> r_{1}\}$. Take conductivity of $\Omega_k$ be $\sigma_k$ and $\sigma_{0}=1$. Thanks for the orthogonality of harmonic functions, there holds
	\begin{equation}
		\begin{aligned}
			&M_{nm}^{cs}=M_{nm}^{sc}=0\quad \text{for all } m,n,\\
			&M_{nm}^{cc}=M_{nm}^{ss}=0\quad \text{for } m\neq n,
		\end{aligned}\notag
	\end{equation}
	and
	\begin{equation}
		M_{nn}^{cc}=M_{nn}^{ss}=0\quad\text{for all } n. \notag
	\end{equation}
	Therefore, the electric potential is perturbed mildly with $N$-GPTs vanishing structure, that is
	\begin{equation}
		(u-H)(x)=\mathcal{O}(|x|^{-N-1})\quad \text{as }|x|\to +\infty. \notag
	\end{equation}
 In what follows, define
 $$\eta_k=\frac{\sigma_k-\sigma_{k-1}}{\sigma_k+\sigma_{k-1}}, \quad k=1,2,\dots,N+1.$$
  For the sake of simplicity, we take $M_n[\eta]=M_{nn}^{cc}$ generated by $\eta=(\eta_1,\dots,\eta_{N+1})$ and $\mathcal{M}_N(\eta)=(M_1[\eta],\dots,M_N[\eta])$. We are mainly concerned with solution to the following nonlinear equations
	\begin{equation}
		\mathcal{M}_N(\eta)=0,\quad \eta\in[-1,1]^{N+1},\ \eta\neq 0. \notag
	\end{equation}
It is readily seen that $\eta=0$ is a trivial solution to the above nonlinear equations. The existence and $N$-GPTs vanishing structure can be shown as follows.
\begin{theorem}\label{fixed core}
		For any given radius $0<r_{N+1}<r_{N}<\dots<r_{2}<r_{1}$ and fixed $\sigma_{N+1}\geq 0$, there exists a combination $\eta=(\eta_1,\dots,\eta_{N+1})\in [-1,1]^{N+1}$ such that the $N$-GPTs $\mathcal{M}_N$ vanish.
\end{theorem}

	\subsection{Three dimensional case}
	The harmonic function $H(x)$ in $\mathbb{R}^3$ can be denoted by
	\begin{equation}
		H(x)=H(0)+\sum_{n=1}^{+\infty}\sum_{n'=-n}^na_n^{n'}r^nY^{n'}_n(\theta,\phi), \notag
	\end{equation}
	where $Y_n^{n'}(\theta,\phi)$ are the $n'$-th spherical harmonics of order $n$. Similarly the fundamental solution $G(x-y)$ has the following expression as $|x|\to \infty$
	\begin{equation}\label{G2}
		G(x-y)=\sum_{n=1}^{+\infty}\sum_{n'=-n}^n \frac{r_y^n}{(2n+1)r_x^{n+1}}Y_n^{n'}(\theta_x,\phi_x)\overline{Y_n^{n'}(\theta_y,\phi_y)}+C,
	\end{equation}
	with the spherical coordinate $(r_x,\theta_x,\phi_x)$. For multi-index $|\alpha|=n$, we define $(a_{\alpha}^{n,m})$ by
	\begin{equation}
		\sum_{|\alpha|=n}a_\alpha^{n,m}x^\alpha=r^nY_n^m(\theta,\phi), \notag
	\end{equation}
	thus the contracted Generalized Polarization Tensors(CGPTs) in three-dimension are then defined as
	\begin{equation}
		M_{mn}^{m'n'}:=\sum_{|\alpha|=m}^{+\infty}\sum_{|\beta|=n}^{+\infty}a_{\alpha}^{m,m'}a_{\beta}^{n,n'}M_{\alpha\beta}, \quad |m'|\leq m,|n'|\leq n, \notag
	\end{equation}
	which leads to the following far-field expression
	\begin{equation}\label{far-field2}
		(u-H)(x)=\sum_{m=1}^{\infty}\sum_{m'=-m}^{m}\frac{Y_m^{m'}(\theta,\phi)}{(2m+1)r^{m+1}}\sum_{n=1}^{\infty}\sum_{n'=-n}^{n}M_{mn}^{m'n'}a_n^{n'},
	\end{equation}
	uniformly as $|x|\to +\infty$.
	
	We shall keep the geometric and conductivity settings of $\Omega$ in accordance with the two dimensional case. We denote $\tilde{M}_n=M_{nn}^{n'n'}$, since $M_{nm}^{n'm'}=0$ for $n\neq m$ or $n'\neq m'$ and $M_{nn}^{n_1'n_1'}=M_{nn}^{n_2'n_2'}$ for $|n_1'|\leq n,\ |n_2'|\leq n$. Thus the $N$-GPTs vanishing structure $\Omega$ can be designed by taking $\tilde{M}_n=0,\ n=1,\dots,N$. Let $\tilde{\mathcal{M}}_N(\eta)=(\tilde{M}_1[\eta],\dots,\tilde{M}_N[\eta])$ be generated by $\eta=(\eta_{1},\dots,\eta_{N+1})$. We shall now prove the following existence result for three dimensional case.

	\begin{theorem}\label{th:main2}
	For any given radius $0<r_{N+1}<r_{N}<\dots<r_{2}<r_{1}$ and fixed $\sigma_{N+1}\geq 0$, there exists a combination $\eta=(\eta_{1},\dots,\eta_{N+1})\in [-1,1]^{N+1}$ such that the $N$-GPTs $\tilde{\mathcal{M}}_N$ vanish.
	\end{theorem}
\section{The existence and uniqueness of $N$-GPTs vanishing structure locally}
\quad \ \ In this section, we shall prove the existence and uniqueness of $N$-GPTs vanishing structure under some assumptions on the solutions or the structure. The related results shall be used to derive our final results in the last section.

Let us first recall the result obtained in \cite{FD23,kong2024inverse} for radially symmetric structure. It can be explicit given that
$M_n=-2\pi n b_0^{(n)}$, where $b_0^{(n)}=e^T \Upsilon_{N+1}^{(n)} (P_{N+1}^{(n)})^{-1} e$, $e=(1,1,\cdots,1)^T$,
	\begin{equation}\label{PN}
		P_{N+1}^{(n)}[\eta]:=\begin{bmatrix}
			1/\eta_1 & (r_2/ r_1)^{2n}  & \cdots & (r_{N+1}/ r_1)^{2n} \\
			-1 & 1/\eta_2  & \cdots &(r_{N+1}/ r_2)^{2n} \\
			\vdots & \vdots  & \ddots & \vdots \\
			 -1 & -1 & \cdots & 1/\eta_{N+1}
		\end{bmatrix},
\end{equation}
 and
\begin{equation}\label{RN}
	\Upsilon_{N+1}^{(n)}:=\begin{bmatrix}
		r_1^{2n} & 0  & \cdots & 0\\
		0 & r_2^{2n}  & \cdots & 0\\
		\vdots & \vdots  & \ddots & \vdots\\
		0 & 0 & \cdots & r_{N+1}^{2n}
	\end{bmatrix}.
	\end{equation}
In three dimensional case, one similarly has $\tilde{M}_n=-(2n+1)e^T \tilde{\Upsilon}_{N+1}^{(n)} (\tilde{P}_{N+1}^{(n)})^{-1} e$, where
	\begin{equation}
		\tilde{P}_{N+1}^{(n)}[\eta]:=\begin{bmatrix}
			\frac{2n+1}{2n}1/\eta_1+\frac{1}{2n} & \frac{n+1}{n}(r_2/ r_1)^{2n+1}  & \cdots & \frac{n+1}{n}(r_{N+1}/ r_1)^{2n+1} \\
			-1 & \frac{2n+1}{2n}1/\eta_2+\frac{1}{2n}  & \cdots &\frac{n+1}{n}(r_{N+1}/ r_2)^{2n+1} \\
			\vdots & \vdots  & \ddots & \vdots \\
			 -1 & -1 & \cdots & \frac{2n+1}{2n}1/\eta_{N+1}+\frac{1}{2n}
		\end{bmatrix},
\end{equation}
and
\begin{equation}
\tilde{\Upsilon}_{N+1}^{(n)}:=\begin{bmatrix}
		r_1^{2n+1} & 0  & \cdots & 0\\
		0 & r_2^{2n+1}  & \cdots & 0\\
		\vdots & \vdots  & \ddots & \vdots\\
		0 & 0 & \cdots & r_N^{2n+1}
	\end{bmatrix}.\notag
	\end{equation}
We want to point out that, in the above setup, we have the assumptions that $\eta_n\neq 0$, $n=1,2,\ldots, N+1$. Actually, if there holds $\eta_{N_0}=0$ for some $N_0\geq 1$ and $\eta_{n}\neq 0$ for $n\neq N_0$, then the $N$ coatings is essentially degenerated to $N-1$ coatings. It is a common sense that $N-1$ coatings can not achieve $N$-GPTs vanishing, at least if the structures (radius) of $N-1$ coatings are fixed. We shall show that this is actually true. But before this, we need to show some locally existence and uniqueness results for $N$-GPTs vanishing with $N$ coatings.
In what follows, we shall define $f_0:=\sum_{n=1}^{N+1} \eta_n$. Numerical observation in \cite{ammari2013enhancement1} show that when $N$-GPTs vanish, $f_0$ is small, which means that the conductivity distributions in the multi-coatings are oscillating. So it is nature to first assume that $f_0$ is small enough.
\subsection{Existence and uniqueness for $f_0$ is small}
\begin{theorem}\label{th:mainsmall1}
Suppose $f_0=\sum_{n=1}^{N+1} \eta_n\ll 1$ is fixed
then there exists a unique solution $\eta$ such that the $N$-GPTs vanish.
\end{theorem}
\begin{proof} We only prove the two dimensional case, while it is similar to show the proof for three dimensional case. Let
 \beq
 \GL:= \begin{bmatrix}
 \eta_1 & 0 & \cdots & 0 \\
 0 & \eta_2 & \cdots & 0 \\
 \vdots & \vdots & \ddots & \vdots \\
 0 & 0 & \cdots & \eta_{N+1}
 \end{bmatrix} \quad\mbox{and}\quad \Pi_n = P_{N+1}^{(n)} - \GL^{-1}.
 \eeq
Then we have
$$
(P_{N+1}^{(n)})^{-1}= (I+\GL\Pi_n)^{-1} \GL.
$$
We now write $b_1^{(n)}$ as
\beq\label{eq:comfin2}
b_1^{(n)}= e^T \Upsilon_{N+1}^{(n)} \GL e + e^T \Upsilon_{N+1}^{(n)}
[(I+\GL\Pi_n)^{-1}-I] \GL e.
\eeq
One can easily see that
\beq
e^T \Upsilon_{N+1}^{(n)} \GL e = \sum_{j=1}^{N+1} \eta_j r_j^{2k},
\eeq
Define
\beq\label{eq:dfnwk}
w_n(\eta):= e^T \Upsilon_{N+1}^{(n)}
[(I+\GL\Pi_n)^{-1}-I] \GL e
\eeq
then there satisfies
\beq\label{wkest}
|w_n(\eta)| \le C |\eta|^2,
\eeq
where $C$ depends on $r_1, \ldots, r_{N+1}$ and $N$.
Thus $\eta=(\eta_1, \ldots, \eta_{N+1})$ satisfies the following equation:
 \beq
 V \eta^T + W(\eta)= \begin{bmatrix} \sum_{j=1}^{N+1} \eta_j \\ b_{1}^{(1)} \\
\vdots \\  b_{1}^{(N)} \end{bmatrix},
 \eeq
where
 $$
 V= \begin{bmatrix}
1 & 1 & \cdots & 1 \\
r_1^2 & r_2^2 & \cdots & r_{N+1}^2 \\
\vdots & \vdots & \ddots & \vdots \\
r_1^{2N} & r_2^{2N} & \cdots & r_{N+1}^{2N}
\end{bmatrix}, \quad
W(\eta) = \begin{bmatrix} 0 \\ w_1(\eta) \\ \vdots \\ w_N(\eta)
\end{bmatrix}.
 $$

Note that $V$ is the Vandermonde matrix and invertible. By \eqnref{wkest} there
are $\Ge>0$ and $C_1, C_2$ depending on $r_1, \ldots, r_{N+1}$ and $N$ such that if
$|\eta| \le \Ge$, then
 \beq
 C_1 |\eta| \le |V \eta + W(\eta)| \le C_2 |\eta|.
 \eeq
By fixed point theory (see, for example, \cite{Ze86}), there is also $\delta_0>0$ such that for all $\mathbf{f}\in \mathbb{R}^{N+1}$ satisfying $|\mathbf{f}| \le \delta_0$
there is a unique solution $\eta$ satisfying $|\eta| \le \Ge$ to
\beq\label{VWf}
V \eta + W(\eta) =\mathbf{f}.
\eeq
In particular, if we take $\mathbf{f}=(f_0, 0, \ldots, 0)^T$ with $|f_0| \le \delta_0$,
then there is a unique solution to \eqnref{VWf}.
\end{proof}

\begin{corollary}
If $\Gs_{N+1}$ is fixed and is close enough to the background conductivity $\Gs_0$ then $N$-GPTs vanishing structure exists and is unique.
\end{corollary}
\begin{proof} For the sake of simplicity, assume that the background conductivity $\Gs_{0}=1$. Then there holds
\beq\label{eq:Gsb_rela1}
\Gs_{N+1}=\prod_{n=1}^{N+1}\frac{1+\eta_n}{1-\eta_n}
\eeq
Since
$\|\eta\|_{l^{\infty}}<1$ by Taylor expansion we have
\beq
\Gs_{N+1}=\prod_{n=1}^{N+1}\left[(1+\eta_n)\sum_{k=0}^{\infty}\eta_n^k\right]=1+2\sum_{n=1}
^{N+1}\eta_n + w_0(\eta)
\eeq
where the residual term $w_0(\eta)$ satisfies
\beq
|w_0(\eta)|\leq C |\eta|^2
\eeq
for $|\eta|\leq\epsilon$. Thus if we take $\mathbf{f}=((1-\Gs_{N+1})/2, 0, \ldots, 0)^T$
and substitute the first term in $W(\eta)$, appeared in Theorem \ref{th:mainsmall1}, with $w_0(\eta)$, then
by following the same way in Theorem \ref{th:mainsmall1} there is a unique solution to \eqnref{VWf} for $|(1-\Gs_{N+1})/2|\leq\delta_0$
for some small $\delta_0>0$.
\end{proof}
\subsection{Existence and uniqueness for $r_{N+1}$ is small}
Next, we shall show that the existence and uniqueness can also be proven in the case that the radius of the core, i.e. $r_{N+1}$ is small enough. The result shall be used for
existence of $N$-GPTs in general situation.

\begin{theorem}\label{th:mainle2}
Suppose that $r_{N+1}\ll 1$, then for any fixed core conductivity $\Gs_{N+1}$, the $N$-GPTs vanishing structure exists and is unique.
\end{theorem}
\begin{proof} Note that from \eqnref{eq:Gsb_rela1}, it is equivalent to fix $\Gs_{N+1}$ and $\eta_{N+1}$. Let
\beq
\tilde{V}=\left[
\begin{array}{cccc}
r_1^2 & r_2^2 & \cdots & r_{N}^2 \\
r_1^4 & r_2^4 & \cdots & r_{N}^4 \\
\vdots & \vdots & \ddots & \vdots \\
r_1^{2N} & r_2^{2N} & \cdots & r_{N}^{2N}
\end{array}
\right],\quad
\tilde{W}(\eta) = \begin{bmatrix} w_1(\eta) \\ w_2(\eta) \\ \vdots \\
w_N(\eta) \end{bmatrix}
\eeq
and $\eta'=(\eta_1,\ldots,\eta_{N})$. Here $w_n(\eta)$ is defined in \eqnref{eq:dfnwk}. Let $b_{0}^{(n)}=0$ then by
minor adjustment on \eqnref{VWf} we have
\beq\label{VWf1}
\tilde{V} {\eta'}^T + \tilde{W}(\eta) =\tilde{\mathbf{f}}.
\eeq
with $\tilde{\mathbf{f}}=-\eta_{N+1}(r_{N+1}^2,r_{N+1}^4,\ldots,r_{N+1}^{2N})^T$. Denote by $Q_{ij}^{(n)}$ the $(i+1)(j+1)$-th element in
$(I+\GL\Pi_n)^{-1}-I$. Since
\begin{eqnarray*}
w_n(\eta) & = & \sum_{i=1}^{N+1}\sum_{j=1}^{N}
r_i^{2n}Q_{ij}^{(n)}\eta_j+\sum_{i=1}^{N+1}r_i^{2n}Q_{i(N+1)}^{(n)}\eta_{N+1} \\
 & = &  \sum_{j=1}^{N}r_{N+1}^{2n}Q_{(N+1)j}^{(n)}\eta_j +
r_{N+1}^{2n}Q_{(N+1)(N+1)}^{(n)}\eta_{N+1}+O(r_{N+1}^{2n})+O(|{\eta'}|^2) \\
 & = & O(r_{N+1}^{2n}) +
O(|\eta'|^2)
\end{eqnarray*}
 By moving the $O(r_{N+1}^{2n})$ part in $w_n$ to the right side of \eqnref{VWf1} one can get that the right hand
side is of order $r_{N+1}^{2n}$ for the $n$-th equation, $n=1,2,\ldots, N$. Since
$\tilde{V}$ is invertible and does not depend on $r_{N+1}$, for any $\Ge>0$, there exists $\delta_1>0$
such that for $r_{N+1}$ satisfying $r_{N+1} \le \delta_1$ there is a unique solution
$\eta'$ satisfying $|\eta'| \le \Ge$ to \eqnref{VWf1}.
\end{proof}

\section{Proof of main theorems}
	\label{sec:3}
	\quad \ \ Let us first rewrite the integral system \eqref{integral} as follows
	\begin{equation}\label{integral2}
		\begin{bmatrix}
			I-2\eta_1\mathcal{K}^{*}_{\Gamma_1} & -2\eta_1\mathcal{K}_{\Gamma_2,\Gamma_1} &\cdots & -2\eta_1\mathcal{K}_{\Gamma_N,\Gamma_1}\\
			-2\eta_2\mathcal{K}_{\Gamma_1,\Gamma_2} & I-2\eta_2\mathcal{K}^{*}_{\Gamma_2} &\cdots & -2\eta_2\mathcal{K}_{\Gamma_N,\Gamma_2}\\
			\vdots & \vdots  & \ddots & \vdots\\
			-2\eta_N\mathcal{K}_{\Gamma_1,\Gamma_N} & -2\eta_N\mathcal{K}_{\Gamma_2,\Gamma_N}  &\cdots & I-2\eta_N\mathcal{K}^{*}_{\Gamma_N}
		\end{bmatrix}\begin{bmatrix}
			\phi_1\\\phi_2\\ \vdots \\\phi_N
		\end{bmatrix}=\begin{bmatrix}
			2\eta_1\partial_{\nu_1}H\\2\eta_2\partial_{\nu_2}H\\ \vdots \\2\eta_N\partial_{\nu_N}H
		\end{bmatrix}.
	\end{equation}
	Denote the coefficient matrix by $\mathbb{E}^{\eta}_{\Omega}$, we shall demonstrate that the CGPTs are continuously differentiable respect to $\eta=(\eta_1,\eta_2,\dots,\eta_N)$.
	\begin{theorem}\label{the1}
		For $\alpha,\beta\in\mathbb{N}^d$ and $\eta_k\in[-1,1],k=1,\dots,N$. Let $\phi^{\eta}_{k,\beta}$ be the solution of
		\begin{equation}
			\mathbb{E}^{\eta}_{\Omega}(\phi_{1,\beta},\dots,\phi_{N,\beta})^T=(2\eta_1\partial_{\nu_1}y^{\beta},2\eta_2\partial_{\nu_2}y^{\beta},\dots,2\eta_N\partial_{\nu_N}y^{\beta})^T.\notag
		\end{equation}
		Then the GPTs $M_{\alpha\beta}$ defined by \eqref{GPT} are continuously differentiable in $\eta_k\in[-1,1]$ and so are the CGPTs.
		\begin{proof}
			We first claim the integral operator $\mathbb{E}^{\eta}_{\Omega}$ is invertible for all $\eta\in[-1,1]^N$. If $\eta_k\neq 0$ for any $k\leq N$, that case can be immediately proved by the invertibility of $\mathbb{J}_{\Omega}$. Let's assume $\eta_{k_i}=0$ for $i=1,2,\dots,s$, the $k_i$-th equation turns to
			\begin{equation}
				I[\phi_{1,\beta}]=0,\quad i=1,2,\dots,s, \notag
			\end{equation}
			thus $\phi_{1,\beta}=0$, this implies the system has degraded to $(N-s)$-layer transmission problem, whose invertibility can be obtained by Theorem \ref{invertible}.
			
			Let $\delta\ll 1$, $\tilde{\eta}=U(\eta,\delta)\cap [-1,1]^N$ and $\Delta{\eta}=\tilde{\eta}-\eta$. Direct computation shows
			\begin{equation}
				\begin{aligned}
					(\mathbb{E}_{\Omega}^{\eta}+\tilde{\mathbb{E}}_{\Omega})(\tilde{\phi}_{1,\beta},\dots,\tilde{\phi}_{N,\beta})^T&=(2\eta_1\partial_{\nu_1}y^{\beta},\dots,2\eta_N\partial_{\nu_N}y^{\beta})^T\\&+(2\Delta\eta_1\partial_{\nu_1}y^{\beta},\dots,2\Delta\eta_N\partial_{\nu_N}y^{\beta})^T,
				\end{aligned} \notag
			\end{equation}
			where
			\begin{equation}
				\tilde{\mathbb{E}}_{\Omega}=\begin{bmatrix}
					-2\Delta{\eta_1}\mathcal{K}^{*}_{\Gamma_1} & -2\Delta{\eta_1}\mathcal{K}_{\Gamma_2,\Gamma_1} &\cdots & -2\Delta{\eta_1}\mathcal{K}_{\Gamma_N,\Gamma_1}\\
					-2\Delta{\eta_2}\mathcal{K}_{\Gamma_1,\Gamma_2} & -2\Delta{\eta_2}\mathcal{K}^{*}_{\Gamma_2} &\cdots & -2\Delta{\eta_2}\mathcal{K}_{\Gamma_N,\Gamma_2}\\
					\vdots & \vdots  & \ddots & \vdots\\
					-2\Delta{\eta_N}\mathcal{K}_{\Gamma_1,\Gamma_N} & -2\Delta{\eta_N}\mathcal{K}_{\Gamma_2,\Gamma_N}  &\cdots & -2\Delta{\eta_N}\mathcal{K}^{*}_{\Gamma_N}
				\end{bmatrix}. \notag
			\end{equation}
			It follows $(\tilde{\phi}_{1,\beta},\dots,\tilde{\phi}_{N,\beta})=(\phi_{1,\beta},\dots,\phi_{N,\beta})+(\phi^{(1)}_{1,\beta},\dots,\phi^{(1)}_{N,\beta})+\mathcal{O}(\delta^2)$, where
			\begin{equation}\label{MD}
				\begin{aligned}
					(\phi^{(1)}_{1,\beta},\dots,\phi^{(1)}_{N,\beta})^T&=-(\mathbb{E}_{\Omega}^{\eta})^{-1}\tilde{\mathbb{E}}_{\Omega}(\mathbb{E}_{\Omega}^{\eta})^{-1}(2\eta_1\partial_{\nu_1}y^{\beta},\dots,2\eta_N\partial_{\nu_N}y^{\beta})^T\\&+(\mathbb{E}_{\Omega}^{\eta})^{-1}(2\Delta\eta_1\partial_{\nu_1}y^{\beta},\dots,2\Delta\eta_N\partial_{\nu_N}y^{\beta})^T.
				\end{aligned}
			\end{equation}
			Since $(\mathbb{E}_{\Omega}^{\eta})^{-1}$ is independent of $\Delta{\eta}$, the right-hand side of \eqref{MD} is linear with $\Delta{\eta}$.
			Plugging \eqref{MD} into \eqref{GPT}, we obtain
			\begin{equation}
				\begin{aligned}
					\tilde{M}_{\alpha\beta}&=\sum_{k=1}^N\int_{\Gamma_k}y^{\alpha}\tilde{\phi}_{k,{\beta}}d\sigma_y\\&=\sum_{k=1}^N\int_{\Gamma_k}y^{\alpha}{\phi}_{k,{\beta}}d\sigma_y+\sum_{k=1}^N\int_{\Gamma_k}y^{\alpha}{\phi}^{(1)}_{k,{\beta}}d\sigma_y+\mathcal{O}(\delta^2)\\
					&:=M_{\alpha\beta}+\mathbf{M}^{(1)}_{\alpha\beta}(\Delta\eta_1,\dots,\Delta\eta_N)^T+\mathcal{O}(\delta^2),
				\end{aligned}\notag
			\end{equation}
			where	$\mathbf{M}^{(1)}_{\alpha\beta}=(m^1_{\alpha\beta},\dots,m^N_{\alpha\beta})$
			is determined by \eqref{MD}. The continuous differentiability of CGPTs follows from \eqref{CGPT1}-\eqref{CGPT4}.
		\end{proof}
	\end{theorem}

	Before proving our main result on the existence of $N$-GPTs vanishing structure, we introduce the Brouwer degree theory, which mainly follows from the monograph \cite{ciarlet2013linear}.
	\begin{definition}\label{deg}
		Let $D\subset \mathbb{R}^n$ be a bounded open set, $\phi \in C^1(\overline{D},\mathbb{R}^n),\ P$ is a regular value of $\phi$ and $P\notin \phi(\partial D)$, define the Brouwer degree in $D$ respect to $P$ as
		\begin{equation}\label{Brouwer degree}
			\text{deg}(\phi,D,P):=\sum_{x\in\phi^{(-1)}(P)}\text{sgn}J_{\phi}(x),
		\end{equation}
		where $\phi^{(-1)}(P)=\{x\in D:\phi(x)=P\}$ and $J_{\phi}(x)=\text{det}\nabla\phi(x)$.
	\end{definition}
	Note that the Brouwer degree can be extended to any $P\notin\phi(\partial D)$ and continuous functions through an integral formulation. However, for simplicity and clarity of computation, we use the definition in Definition \ref{deg}. Moreover, one of the most important properties of Brouwer degree is the homotopic invariance:
	\begin{theorem}\label{th:thmin}{\normalfont[Theorem 9.15-6 in \cite{ciarlet2013linear}]}
		Let $D\subset \mathbb{R}^n$ be a bounded open set, $f$, $g\in C(\overline{D},\mathbb{R}^n)$ and a homotopy $H\in C(\overline{D}\times [0,1],\mathbb{R}^n)$ joining $f$ to $g$ in the space $C(\overline{D},\mathbb{R}^n)$, that is
		\begin{equation}
			H(\cdot,0)=f\quad \text{and} \quad H(\cdot,1)=g. \notag
		\end{equation}
		If the point $b\notin H(\partial D\times [0,1])$,
		then
		\begin{equation}
			\mathrm{deg}(H(\cdot,\lambda),D,b)=\mathrm{deg}(f,D,b)\quad \text{for all }0\leq \lambda\leq1. \notag
		\end{equation}
		In particular, $\mathrm{deg}(f,D,b)=\mathrm{deg}(g,D,b)$.
	\end{theorem}

	Specially, the following theorem investigates that the Brouwer Degree of a continuous odd map is nonzero.
	\begin{theorem}\label{Borsuk-Ulam}{\normalfont[Theorem 9.17-2 in \cite{ciarlet2013linear}]}
		 Let $D\subset \mathbb{R}^n$ be a bounded  origin-symmetric open set such that $0\in D$, $\phi\in C(\overline{D}\times [0,1],\mathbb{R}^n)$ and $\phi(x)=-\phi(-x)$ on $\partial D$. Then,  the Brouwer Degree $\mathrm{deg}(f,D,0)$ is odd.
	\end{theorem}
	The Brouwer Degree Theory is usually used to prove solvability, in fact, the solvability theorem holds.
	\begin{theorem}\label{solvability}{\normalfont[Theorem 9.15-4 in \cite{ciarlet2013linear}]}
		Let $D\subset \mathbb{R}^n$ be a bounded open set, $\phi\in C(\overline{D},\mathbb{R}^n)$ and $b\in \phi(\partial D)$. If $b \notin \phi(\overline{D})$, then
		\begin{equation}
			\mathrm{deg}(f,D,b)=0. \notag
		\end{equation}
		Hence,
		\begin{equation}
			\mathrm{deg}(f,D,b) \notin 0 \text{ implies that } f(x)=b \text{ for some } x\in D. \notag
		\end{equation}
	\end{theorem}

We establish the existence of $N$-GPTs vanishing structures by employing the Brouwer Degree theory:
	\begin{theorem}\label{main result}
		For any given radius $0<r_{N+1}<r_{N}<\dots<r_{2}<r_{1}$ and fixed $\eta_{N+1}\in [-1,1]$, there exists a nontrivial combination $\eta'=(\eta_1,\dots,\eta_N)$ with $\eta_n\in [-1,1],\ n=1,\dots,N$ such that $N$-GPTs vanish, i.e. $\mathcal{M}_N(\eta',\eta_{N+1})=0$.
		\begin{proof}
			Let $D=\{\eta' : |\eta_n|\leq 1, n=1,\dots,N\}$, and define the map $F_t(\eta')=\mathcal{M}_N(\eta',t),\ | t|\leq1$. It follows from Theorem \ref{the1} that $F_t$ is a homotopy joining $F_{-1}$ to $F_1$.
			
			Firstly, we shall point out that if $\eta'\in \partial D$, then there holds
			\begin{equation}\label{independent}
				\mathcal{M}_N(\eta',t_1)=\mathcal{M}_N(\eta',t_2) \quad\text{for any }t_1\neq t_2.
			\end{equation}
			In fact,  if $\eta'\in \partial D$ there exists some $1\leq p\leq N$ such that $\eta_p=\pm 1$. It suffices to consider $\eta_p=-1$, as the other case is similar.
			
			Consider the following transmission system:
			\begin{equation}\label{ntransmission}
				\begin{cases}
					\Delta u=0,& \text{in }\mathbb{R}^2\backslash\cup_{k=1}^{N+1}\Gamma_k,\\
					u|_{+}=u|_{-}, & \text{on } \Gamma_k, \  k>p,\\
					(1-\eta_k)\partial_{\nu}u|_{+}=(1+\eta_k)\partial_{\nu}u|_{-}, & \text{on } \Gamma_k,\ k> p,\\
					\partial_{\nu}u|_{+}=0, & \text{on } \Gamma_p,\\
					u-r^n\cos n\theta=\mathcal{O}(|x|^{-1}), &|x|\to \infty,
				\end{cases}
			\end{equation}
			The orthogonality of harmonic functions shows
			\begin{equation}
				u-r^n\cos n\theta=\frac{b_n\cos n\theta}{r^n} \quad \text{in } \mathbb{R}^2\backslash \Omega,
			\end{equation}
			thus
			\begin{equation}
				M_n=-2\pi nb_n,
			\end{equation}
			follows from \eqref{far-field}. Note that \eqref{ntransmission} can be treated as an exterior Neumann boundary problem, therefore $M_n$ is independent of $\eta_{N+1}$, which implies \eqref{independent}.
			
Next we suppose that $0\notin F_0(\partial D)$, which will be verified. In this case, by \eqref{independent} one has $0\notin F_t(\partial D)$.
			Now let $u_k^{\perp}$ be the harmonic conjugate of $u$ in each $\Omega_k$ and define $v=\sigma_k u_k^{\perp}+c_k$ in $\Omega_k$, where
			\begin{equation}
				\sigma_k=\prod_{i=1}^{k}\frac{1+\eta_i}{1-\eta_i}\quad k=0,1,\dots,N+1, \notag
			\end{equation}
			and the constants $c_k$ are chosen such that $v$ satisfies the transmission conditions across $\partial\Omega_k$, $k=0,1, \ldots, N+1$. Then $v$ satisfies the following equation
			\begin{equation}\label{ntransmission2}
				\begin{cases}
					\Delta v=0,& \text{in }\mathbb{R}^2\backslash\cup_{k=1}^{N+1}\Gamma_k,\\
					v|_{+}=v|_{-}, & \text{on } \Gamma_k,\\
					(1+\eta_k)\partial_{\nu}v|_{+}=(1-\eta_k)\partial_{\nu}v|_{-}, & \text{on } \Gamma_k,\\
					v+r^n\sin n\theta=\mathcal{O}(|x|^{-1}), &|x|\to \infty,
				\end{cases}
			\end{equation}
			therefore we can immediately deduce that $\mathcal{M}(\eta',t)=-\mathcal{M}(-\eta',-t)$ and hence $F_0(\eta')=-F_0(-\eta')$. It follows from Theorem \ref{Borsuk-Ulam} that
			\begin{equation}
				\text{deg}(F_0,D,0)\neq 0. \notag
			\end{equation}

			By the homotopy invariance of Brouwer degree (Theorem \ref{th:thmin}), we deduce that
			\begin{equation}\label{brouwer degree}
				\text{deg}(F_t,D,0)=\text{deg}(F_0,D,0)\neq 0,
			\end{equation}
			for which we can immediately get the existence of $N$-GPTs vanishing structure by Theorem \ref{solvability}.

Finally we show that $0\notin F_0(\partial D)$, otherwise by \eqref{independent}, there exist $\eta'\in \partial D$ such that
			\begin{equation}
			F_t(\eta')=F_0(\eta')=0 \quad \text{for all }|t|\leq 1. \notag
			\end{equation}
Since in this case $\eta_{N+1}=0$, the structure is equivalent to $N$-coatings over $\Omega_{N+1}$ with radius $r_{N+1}$ sufficiently small. Thus there exist at least two solutions to such $N$-GPTs vanishing structure, one is at condition $\eta_{N+1}=0$ while the other is at $\eta_{N+1}\neq 0$. This is contradiction by using Theorem \ref{th:mainle2}.
		\end{proof}
	\end{theorem}

		\begin{proof}[Proof of Theorem \ref{fixed core}]
			Without loss of generality we assume $\sigma_0=1$, there holds
			\begin{equation}
				\sigma_{N+1}=\prod_{i=1}^{N+1}\frac{1+\eta_i}{1-\eta_i}=t.
			\end{equation}
			Let $\overline{M}_t^{(n)}=M_n[\eta_1,\dots,\eta_{N+1}]$, where
			\begin{equation}
				\eta_{N+1}=\frac{t\prod_{i=1}^N(1-\eta_i)-\prod_{i=1}^N(1+\eta_i)}{t\prod_{i=1}^N(1-\eta_i)+\prod_{i=1}^N(1+\eta_i)}, \notag
			\end{equation}  and define $G_t(\eta_1,\dots,\eta_N)=(\overline{M}_t^{(1)},\dots,\overline{M}_t^{(N)})$. For $t=0$, we can immediately obtain $\eta_{N+1}=-1$, whose existence is solved in Theorem \ref{main result}. The Brouwer degree of $G_0$ in $D$ follows from \eqref{brouwer degree} that
			\begin{equation}
				\text{deg}(G_0,D,0)\neq 0. \notag
			\end{equation}
			Using the same approach in Theorem \ref{main result}, the homotopy invariance of Brouwer degree holds. It follows that
			\begin{equation}
				\text{deg}(G_t,D,0)=\text{deg}(G_0,D,0)\neq 0.
			\end{equation}
			Then the existence of $N$-GPTs vanishing structure proved.
		\end{proof}

		\begin{proof}[Proof of Theorem \ref{th:main2}]
			Let us first define $D=\{\eta' : |\eta_n|\leq 1, n=1,\dots,N\}$ and $\tilde{F}_{t}(\eta')=\tilde{\mathcal{M}}_N(\eta',t)$. Note that we can not use the same harmonic conjugate arguments to derive that $\mathcal{M}_N(\eta)$ is a odd map. We shall compute the Brouwer degree of $\tilde{F}_0$ directly.
			
			Let
			\begin{equation}
				u(x)=\sum_{n=0}^N H_n(x)+\sum_{k=1}^{N+1}\mathcal{S}_{\Gamma_k}[\phi_k](x),\quad x\in\mathbb{R}^3, \notag
			\end{equation}
			with $H_n=r^nY_n^{n'}(\theta,\phi)$ and $(\phi_1,\dots,\phi_{N+1})$ satisfy \eqref{integral2}. Thus
			\begin{equation}
				(u-\sum_{n=0}^N H_n)(x)=\sum_{n=0}^N\frac{b_nY_n^{n'}(\theta,\phi)}{r^{n+1}} \quad\text{in }\mathbb{R}^3\backslash \Omega, \notag
			\end{equation} and $\tilde{M}_n=(2n+1)b_n$ follows from \eqref{far-field2}. By using (\ref{MD}) and the fact that (see, e.g., \cite{ammari2013spectral,ammari2013mathematical})
\begin{equation} \label{Single sphere}
\mathcal{S}_{\Gamma_k}[Y_n^{n'}](x) = \begin{cases}
 - \frac{1}{2n+1} \frac{r^n}{r_k^{n-1}} Y_n^{n'} (\hat x), \quad & \mbox{if } |x|=r <  r_k, \\
 - \frac{1}{2n+1} \frac{r_k^{n+2}}{r^{n+1}} Y_n^{n'} (\hat x), \quad & \mbox{if } |x|=r > r_k,
\end{cases}
\end{equation}
we immediately get
			\begin{equation}
				\frac{\partial\tilde{M}_n}{\partial \eta_k}\Big|_{\eta=0}=(2n+1)r^{n+1}\mathcal{S}_{\Gamma_k}[\partial_{\nu_k}H_n]=nr_k^{2n+1}.
			\end{equation}
			Note that $\tilde{F}_0=(\tilde{M}_1,\dots,\tilde{M}_N)$, therefore
			\begin{equation}
				\nabla{\tilde{F}_0}(0)=\begin{bmatrix}
					r_1^{3} & r_2^{3}  & \cdots & r_N^{3} \\
					2r_1^{5} & 2r_2^{5}  & \cdots & 2r_N^{5} \\
					\vdots & \vdots  & \ddots & \vdots \\
					Nr_1^{2N+1} & Nr_2^{2N+1} & \cdots & Nr_N^{2N+1}
				\end{bmatrix},\notag
			\end{equation}
			where the Jacobi matrix $\nabla{\tilde{F}_0}(0)$ is of Vandermonde form which is invertible. Similar to Theorem \ref{main result}, we can consider the case of  $\eta_{N+1}=0$ to $N$-coatings over $\Omega_{N+1}$ with radius $r_{N+1}$ sufficiently small. It follows from Theorem \ref{th:mainle2} there exists a unique $|\eta'|\leq \epsilon$ such that $N$-GPTs vanishing. Therefore one can readily derive that $\tilde{F}_0(\eta')\neq 0$ for $\eta'\neq 0$. Then it follows from \eqref{Brouwer degree} that
			\begin{equation}
				\text{deg}(\tilde{F}_0,D,0)=\text{sgn}J_{\tilde{F}_0}(0)=1.
			\end{equation}
Thus the existence of fixed $\eta_{N+1}$ can be proved by homotopic invariance similar to Theorem \ref{main result}. We have shown that there exists a nontrivial combination $\eta'=(\eta_1,\dots,\eta_N)$ with $\eta_n\in [-1,1],\ n=1,\dots,N$ such that $N$-GPTs vanish. Following the similar steps in Theorem \ref{fixed core}, one can show the existence of $N$-GPTs vanishing structure for any fixed $\Gs_{N+1}$. The proof is complete.
		\end{proof}

	\section{Multi-layer structure with proportional radius}
	\label{sec:proportional}
	\quad\ After resolving the global existence problem, it is natural to consider whether the $N$-GPTs vanishing structure is unique under fixed geometry and core conductivity. It is noted that we have only shown the uniqueness locally in section 4. In this section, we investigate a specific class of $N$-GPTs vanishing structures with special radius-settings and derive a uniqueness result for insulating core. Let $r_k=r_{N+1} \gamma^{N+1-k},k=1,\dots,N+1$ with growth parameter $\gamma>1$, where the radius of layers increase by a constant scale.
	\subsection{A uniqueness result of proportional radius}
	\quad \ To start with, we provide an explicit formula for CGPTs in two dimension. We will continue to use the notation of $M_n[\eta]$ in Section \ref{sec:3}. Let  $H_n=r^ne^{in\theta}$ in polar coordinates $(r,\theta)$. For the orthogonality, we can express the solution $u_n$ in the form
	\begin{equation}
		u_n=a_k^{(n)}r^ne^{in\theta}+\frac{b_k^{(n)}}{r^n}e^{in\theta},\quad \text{in }\Omega_k, \quad k=0,1,\dots,N+1,
	\end{equation}
	where $a_{0}^{(n)}=1$ and $b_{N+1}^{(n)}=0$. Then it follows from \eqref{far-field} that
	\begin{equation}\label{k cgpt}
		M_n=-2\pi nb_0^{(n)}.
	\end{equation}

	From transmission boundary conditions, the following recursion formulas hold
	\begin{equation}
		\sigma_{k}a_k^{(n)}r_k^n+\sigma_{k}\frac{b_k^{(n)}}{r_k^n}=\sigma_{k}a_{k-1}^{(n)}r_k^n+\sigma_{k}\frac{b_{k-1}^{(n)}}{r_k^n},\notag
	\end{equation}
	\begin{equation}
		\sigma_{k}a_k^{(n)}r_k^{n}-\sigma_{k}\frac{b_k^{(n)}}{r_k^n}=\sigma_{k-1}a_{k-1}^{(n)}r_k^n-\sigma_{k-1}\frac{b_{k-1}^{(n)}}{r_k^n},\notag
	\end{equation}
	which can immediately deduce
	\begin{equation}\label{recursion}
		\prod_{i=1}^k\begin{bmatrix}
			2\sigma_i &0 \\
			0&2\sigma_i
		\end{bmatrix}\begin{bmatrix} a^{(n)}_{k} \\ b^{(n)}_{k} \\ \end{bmatrix}=\prod_{i=1}^k\begin{bmatrix}
			\sigma_i+\sigma_{i-1} &0 \\
			0&\sigma_i+\sigma_{i-1}
		\end{bmatrix}
		\begin{bmatrix}
			1 &\eta_i r_i^{-2n} \\
			\eta_i r_i^{2n}&1
		\end{bmatrix}
		\begin{bmatrix} 1 \\ b^{(n)}_{0} \\ \end{bmatrix}.
	\end{equation}

	For the convenience of description, we denote by $\mathcal{C}^{m}_{n}$ the set of all combinations of $m$ out $n$, $m\leq n$ and multi-index $\mathbf{i}=(i_1,i_2,\dots,i_m)\in \mathcal{C}^{n}_{m}$, where the indexes are rearranged as $i_1<i_2<\dots<i_n$. Moreover, $\lceil \cdot \rceil$ and $\lfloor  \cdot \rfloor$ represent rounding-up function and rounding-down functions, respectively.
	\begin{lemma}\label{explicit lemma}
		Let $r_1>r_2>\dots>r_{N+1}>0$ and $\sigma_{i-1}+ \sigma_{i}\neq 0$ for all $i\leq N+1$, the CGPTs $M_n$ generated by N+1-layer structure have the following representation
		\begin{equation}\label{explicit M}
			M_n= \frac{2\pi np_{21}^{(n)}}{p_{22}^{(n)}},\quad n\in\mathbb{N}_{+}.
		\end{equation}
		where
		\begin{equation}\label{p21}
			p_{21}^{(n)}=\sum_{j=1}^{\lceil (N+1)/2\rceil}\sum_{\mathbf{i}\in \mathcal{C}_{N+1}^{2j-1}}\prod_{s=1}^{2j-1}\eta_{i_s}r_{i_s}^{(-1)^{s+1}2n},
		\end{equation}
		\begin{equation}\label{p22}
			p_{22}^{(n)}=1+\sum_{j=1}^{\lfloor (N+1)/2\rfloor}\sum_{\mathbf{i}\in \mathcal{C}_{N+1}^{2j}}\prod_{s=1}^{2j}\eta_{i_s}r_{i_s}^{(-1)^{s+1}2n}.
		\end{equation}
	\end{lemma}
	\begin{proof}
		Let
		\begin{equation}
			P^{(n)}=\begin{bmatrix} p_{11}^{(n)} &p_{12}^{(n)} \\
				p_{21}^{(n)} &p_{22}^{(n)}\end{bmatrix}=
			\prod_{i=1}^{N+1}\begin{bmatrix}
				1 &\eta_i r_i^{-2n} \\
				\eta_i r_i^{2n}&1
			\end{bmatrix}. \notag
		\end{equation}
		From \eqref{recursion}, one immediately gets that
		\begin{equation}
			\prod_{i=1}^{N+1}\begin{bmatrix}
				2\sigma_i &0 \\
				0&2\sigma_i
			\end{bmatrix}\begin{bmatrix} a^{(n)}_{N+1} \\ b^{(n)}_{N+1} \\ \end{bmatrix}=\prod_{i=1}^{N+1}\begin{bmatrix}
				\sigma_i+\sigma_{i-1} &0 \\
				0&\sigma_i+\sigma_{i-1}
			\end{bmatrix}
			\begin{bmatrix} p_{11}^{(n)} &p_{12}^{(n)} \\
				p_{21}^{(n)} &p_{22}^{(n)}\end{bmatrix}
			\begin{bmatrix} 1 \\ b^{(n)}_{0} \\ \end{bmatrix}. \notag
		\end{equation}
		Note that $b_{N+1}^{(n)}=0$, which implies $p_{21}^{(n)}+b_0^{(n)}p_{22}^{(n)}=0$, we conclude that \eqref{explicit M} holds with the help of \eqref{k cgpt}.
		
		We now deduce the explicit representations of $p_{21}^{(n)}$ and $p_{22}^{(n)}$. For simplicity in notation, we let
		\begin{equation}
			\Lambda_i^{(n)}=\begin{bmatrix}
				0 &\eta_i r_i^{-2n} \\
				\eta_i r_i^{2n}&0
			\end{bmatrix},\notag
		\end{equation}
		thus the matrix $P^{(n)}$ turns to
		\begin{equation}
			P^{(n)}=\prod_{i=1}^{N+1}(I+\Lambda_i^{(n)})=\sum_{j=0}^{N+1}\sum_{\mathbf{i}\in \mathcal{C}_{N+1}^{j}}\prod_{s=1}^j \Lambda^{(n)}_{i_s}. \notag
		\end{equation}
		It is easy to see that $\Lambda_i^{(n)}$ is anti-diagonal, we have
		\begin{equation}
			\Lambda^{(n)}_{i_p}\Lambda^{(n)}_{i_q}=\begin{bmatrix}
				\eta_p r_p^{-2n}\eta_q r_q^{2n} & 0 \\
				0&\eta_p r_p^{2n}\eta_q r_q^{-2n}
			\end{bmatrix}. \notag
		\end{equation}
		By simple induction, it follows that for any even $j$,
		\begin{equation}
			\prod_{s=1}^{j}\Lambda^{(n)}_{i_s}=\begin{bmatrix}
				\eta_{i_1} r_{i_1}^{-2n}\eta_{i_2} r_{i_2}^{2n}\dots\eta_{i_{j-1}} r_{i_{j-1}}^{2n}\eta_{i_{j}} r_{i_{j}}^{2n} & 0 \\
				0&\eta_{i_1} r_{i_1}^{2n}\eta_{i_2} r_{i_2}^{-2n}\dots\eta_{i_{j-1}} r_{i_{j-1}}^{-2n}\eta_{i_{j}} r_{i_{j}}^{-2n}
			\end{bmatrix},\notag
		\end{equation}
		and for any odd $j$ there holds
		\begin{equation}
			\prod_{s=1}^{j}\Lambda^{(n)}_{i_s}=\begin{bmatrix}
				0 & \eta_{i_1} r_{i_1}^{-2n}\eta_{i_2} r_{i_2}^{2n}\dots\eta_{i_{j-1}} r_{i_{j-1}}^{2n}\eta_{i_{j}} r_{i_{j}}^{-2n} \\
				\eta_{i_1} r_{i_1}^{2n}\eta_{i_2} r_{i_2}^{-2n}\dots\eta_{i_{j-1}} r_{i_{j-1}}^{-2n}\eta_{i_{j}} r_{i_{j}}^{2n}&0
			\end{bmatrix}.\notag
		\end{equation}
To summarize, we can deduce \eqref{p21} and \eqref{p22}.
	\end{proof}
	\begin{remark}\label{remark1}
		We remark that $p_{21}^{(n)}$ is a polynomial of $\eta$ and the coefficients are determined by the radius of the structure. Specifically, the coefficient of $\eta_{i_1}\cdots\eta_{i_{2j-1}}$ is given by
		\begin{equation}
			r_{i_1}^{2n}r_{i_2}^{-2n}\cdots r_{i_{2j-2}}^{-2n}r_{i_{2j-1}}^{2n}.\notag
		\end{equation} For rearrangement, it's easy to see that
		\begin{equation}
			r_{i_{2j-1}}^{2n}<r_{i_1}^{2n}r_{i_2}^{-2n}\cdots r_{i_{2j-2}}^{-2n}r_{i_{2j-1}}^{2n}<r_{i_1}^{2n}.\notag
		\end{equation}  Under the proportional radius settings, this property will play a crucial role in our subsequent analysis.
	\end{remark}

	\begin{lemma}\label{zero sol}
		Let $\eta=(\eta',\eta_{N+1})$ and $\mathcal{M}(\eta)=(M_1,M_2,\dots,M_N)$, if $r_k=r \gamma^{N+1-k},k=1,\dots,N+1$, then the nonlinear algebra system $\mathcal{M}(\eta',0)=0$ if and only if $\eta'=0$.
		\begin{proof}
			The sufficiency is evident, we shall now prove the necessity. Note that for $\eta_{N+1}=0$, the structure can be considered by $N$ layers.
			
			In order to achieve $M_n=0$, it's sufficient to consider the value of $p_{21}^{(n)}$. Let $\psi_n=p_{21}^{(n)}$, it follows from Lemma \ref{explicit lemma} that
			\begin{equation}\label{phik}
				\psi_{n} = \sum_{j=1}^{\lceil N/2\rceil}\sum_{\text{i}\in C_{N}^{2j-1}}r^{2n}\gamma^{2m_{{\mathbf{i}}}n}\prod_{s=1}^{2j-1}\eta_{i_s},
			\end{equation}
			where
			\begin{equation}\label{mi1}
				m_\mathbf{i}=N+1-i_{1}+i_{2}-\dots+i_{2j-2}-i_{2j-1}.
			\end{equation}
			From rearrangement, i.e. $1\leq i_{1}<i_{2}<\dots<i_{2j-2}<i_{2j-1}\leq N$, it follows that
			\begin{equation}
				1\leq N+1-i_{2j-1}\leq m_{\mathbf{i}}\leq N+1-i_{1}\leq N. \notag
			\end{equation}
			Define
			\begin{equation}\label{Aj1}
				A^N_k:=\{\mathbf{i}\in \mathcal{C}_{N}^{2j-1}|m_\mathbf{i}=k,\ 2j-1\leq N\},\  k=1,2,\dots,N,
			\end{equation} by \eqref{phik} we have
			\begin{equation}
				\psi^{(n)} = \sum_{k=1}^{N}r^{2n}\gamma^{2kn}\sum_{\mathbf{i}\in A^N_k}\prod_{s=1}^{2j-1}\eta_{i_s}.
			\end{equation}
			
			Let $\xi_{k}:=\sum_{\mathbf{i}\in A_{k}^N}\prod_{s=1}^{2j-1}\eta_{i_s}$ and $\xi=(\xi_1,\xi_2,\dots,\xi_{N})$, the nonlinear system $\psi^{(n)} = 0,$ $n=1,2,\dots,N$ turns to a linear algebra system with
			\begin{equation}\label{linear system}
				\begin{bmatrix}
					r^2\gamma^2 & r^2\gamma^2 & r^2\gamma^4 & \cdots & r^2\gamma^{2N}\\
					r^4\gamma^4 & r^4\gamma^4 & r^4\gamma^8 & \cdots & r^4\gamma^{4N}\\
					r^6\gamma^6 & r_N^6\gamma^6 & r^6\gamma^{12} & \cdots & r^6\gamma^{6N}\\
					\vdots & \vdots & \vdots & \ddots & \vdots\\
					r^{2N}\gamma^{2N} & r^{2N}\gamma^{2N} & r^{2N}\gamma^{4N} & \cdots & r^{2N}\gamma^{2NN}
				\end{bmatrix}\begin{bmatrix}
					\xi_1\\\xi_2\\\xi_3\\ \vdots \\\xi_{N}
				\end{bmatrix}=\begin{bmatrix}
					0\\0\\0\\ \vdots \\0
				\end{bmatrix},
			\end{equation}
			which coefficient matrix $M$ is Vandermonde form.
			
			Since the coefficient matrix is non-degenerate, we claim that $\xi=0$ if and only if $\psi^{(n)} = 0$. Note $\xi_1=\eta_N$ and $\xi_N=\eta_1$ so that $\eta_1=\eta_N=0$. More generally, assuming that $\eta_1=\eta_2=\dots=\eta_i=0$, it follows from Remark \ref{remark1} that $\xi_{N-i}=\eta_{i+1}=0$. By induction we can immediately deduce
			\begin{equation}
				\eta_1=\eta_2=\dots=\eta_N=0.
			\end{equation}
			which implies that the nonlinear system $\mathcal{M}(\eta',0)=0$ have only trivial solution.
		\end{proof}
	\end{lemma}
	
We want to mention that the results of Lemma \ref{zero sol} have been proven for general settings of radius. However, the proof is useful for our subsequent uniqueness result.	
	\begin{theorem}\label{pro radii}
		For the radius given by $r_k=r \gamma^{N+1-k}$, $ k=1,\dots,{N+1}$,
		$\gamma>1$ and fixed core conductivity $\sigma_{N+1}\geq 0$, there exists a combination $(\sigma_1,\dots,\sigma_{N})$ such that N-GPTs vanish.  Furthermore, if $\sigma_{N+1}=0$, the combination is unique.
	\end{theorem}
	\begin{proof}
		The existence can be deduced immediately by Corollary \ref{fixed core}. Then we show the uniqueness of solution with $\sigma_{N+1}=0$.
		
		Let
		\begin{equation}
			\xi_{k+1}:=\sum_{\text{i}\in A_{k}^{N+1}}\prod_{s=1}^{2j-1}\eta_{i_s},
		\end{equation}
		where $A_{k}^{N+1}$ is defined by \eqref{mi1} and \eqref{Aj1}. Recalling $\sigma_{N+1}=0$ equals to $\eta_{N+1}=-1$, and thus $\xi_1=\eta_{N+1}=-1$. Using the symmetry of \eqref{p21} and \eqref{p22} we get
		\begin{equation}
			M_n(\eta_1,\eta_2,\dots,\eta_{N+1})=-M_n(-\eta_1,-\eta_2,\dots,-\eta_{N+1}),
		\end{equation}
		it's sufficient to consider the case of  $\eta_{N+1}=1$. Using the same approach in Lemma \ref{zero sol}, it can be shown that
		\begin{equation}\label{linear system2}
			\begin{bmatrix}
				r^2\gamma^2 & r^2\gamma^4 & r^2\gamma^6 &\cdots & r^2\gamma^{2N}\\
				r^4\gamma^4 & r^4\gamma^8 & r^4\gamma^{12} &\cdots & r^4\gamma^{4N}\\
				r^6\gamma^6 & r^6\gamma^{12} & r^6\gamma^{18} &\cdots & r^6\gamma^{6N}\\
				\vdots & \vdots & \vdots & \ddots & \vdots\\
				r^{2N}\gamma^{2N} & r^{2N}\gamma^{4N} & r^{2N}\gamma^{6N} &\cdots & r^{2N}\gamma^{2N^2}
			\end{bmatrix}\begin{bmatrix}
				\xi_2\\\xi_3\\\xi_4\\ \vdots \\\xi_{N+1}
			\end{bmatrix}=-\begin{bmatrix}
				r^2\\r^4\\r^6\\ \vdots \\r^{2N}
			\end{bmatrix}.
		\end{equation}
		Therefore $(\xi_2,\xi_3,\xi_4,\dots,\xi_{N+1})$ is unique determined by the non-trivial linear system \eqref{linear system2}, which implies that $p_{21}^{(n)}$ can be unique determined for any $n\in \mathbb{N}_+$
		\begin{equation}
			p_{21}^{(n)}=\begin{bmatrix}
				r^{2n}&r^{2n}\gamma^{2n}&\cdots& r^{2n}\gamma^{2nN}\end{bmatrix}\begin{bmatrix}
				\xi_1\\\xi_2\\\ \vdots \\\xi_{N+1}
			\end{bmatrix}.
		\end{equation}
		Furthermore, we shall point out $p_{22}^{(n)}$ can also be determined by $(\xi_1,\xi_2,\dots,\xi_{N+1})$. Let $\mathsf{m}_{\mathbf{i}}=-i_1+i_2-\dots-i_{2j-1}+i_{2j}$, where $\mathbf{i}\in C^{2j}_{N+1}$, then we have
		\begin{equation}
			p_{22}^{(n)}=\begin{bmatrix}
				1&\gamma^{2n}&\cdots& \gamma^{2nN}\end{bmatrix}\begin{bmatrix}
				1\\\zeta_1\\\ \vdots \\\zeta_{N}
			\end{bmatrix},
		\end{equation}
		where $\zeta_{k}=\sum_{\text{i}\in \mathcal{A}_k^{N+1}}\prod_{s=1}^{2j}\eta_{i_s}$, $\mathcal{A}_k^{N+1}=
		\{\mathbf{i}\in C_{N}^{2j}|\mathsf{m}_\mathbf{i}=k,2j\leq N+1 \}$.
		
		In terms of combination, we can separate $\xi_{k+1}=\xi_{k+1}^{(1)}+\xi_{k+1}^{(2)}$ with
		\begin{equation}
			\xi_{k+1}^{(1)}=\Big(\sum_{\mathbf{i}\in \mathcal{A}_k^N}\prod_{s=1}^{2j}\eta_{i_s}\Big)\eta_{N+1},
		\end{equation}
		\begin{equation}
			\xi_{k+1}^{(2)}=\sum_{\mathbf{i}\in A_{k}^N}\prod_{s=1}^{2j-1}\eta_{i_s}.
		\end{equation}
	
		Since $\mathbf{i}=(i_1,i_2,\dots,i_{2j-1})\in A_{k}^{N}$ equals to
		\begin{equation}
			N+1-i_{1}+i_{2}+\dots+i_{2j-2}-i_{2j-1}=k, \notag
		\end{equation}
		thus $(i_1,i_2,\dots,i_{2j-1},{N+1})\in \mathcal{A}_k^{N+1}$. Combining with the definition of $\zeta_k$ it follows
		\begin{equation}
			\zeta_k=\frac{\xi_{k+1}^{(1)}}{\eta_{N+1}}+\xi_{k+1}^{(2)}\eta_{N+1},
		\end{equation}
		therefore $\zeta_k=\xi_{k+1}^{(1)}+\xi_{k+1}^{(2)}=\xi_{k+1}$ is unique. The total field $u_n-H_n=b_0^{(n)}e^{in\theta}/r^n=-p_{21}^{(n)}e^{in\theta}/{p_{22}^{(n)}r^n}$ is determined uniquely. Thus the uniqueness of conductivity distribution can be obtained by uniqueness of $L^{\infty}(\Omega)$ under Dirichlet-to-Neumann (DtN) map \cite{astala2006calderon}.
	\end{proof}
	\subsection{The extreme case}
	\label{subsec:extreme}
	\quad \ In this part, we shall consider the distance between two adjacent layers is extremely small. In other words, the core of multi-layer structure is enclosed by extremely thin layers. In mathematically, we set $\gamma=1+\delta$, where $\delta\ll1$. One then has
	\begin{equation}
		r_k=r_{N+1} (1+\delta)^{N-k},\quad k=1,\dots,{N+1}.
	\end{equation}
	Our work focus on the conductivity distribution of $N$-GPTs vanishing structure which has insulative core. Take $\lambda_{N+1}=-1$, similar to \eqref{linear system2} we have
	\begin{equation}\label{extremeeq}
		\begin{bmatrix}
			(1+\delta)^2 & (1+\delta)^4 & (1+\delta)^6 &\cdots & (1+\delta)^{2N}\\
			(1+\delta)^4 & (1+\delta)^8 & (1+\delta)^{12} &\cdots & (1+\delta)^{4N}\\
			(1+\delta)^6 & (1+\delta)^{12} & (1+\delta)^{18} &\cdots & (1+\delta)^{6N}\\
			\vdots & \vdots & \vdots & \ddots & \vdots\\
			(1+\delta)^{2N} & (1+\delta)^{4N} & (1+\delta)^{6N} &\cdots & (1+\delta)^{2N^2}
		\end{bmatrix}\begin{bmatrix}
			\xi_2\\\xi_3\\\xi_4\\ \vdots \\\xi_{N+1}
		\end{bmatrix}=\begin{bmatrix}
			1\\1\\1\\ \vdots \\1
		\end{bmatrix},
	\end{equation}
	whose existence and uniqueness of conductivity distribution has been established by Theorem \ref{pro radii}. The exact value of conductivity is characterized by the following Theorem
	\begin{theorem}
		Assume the unique solution $\eta$ of \eqref{extremeeq} is continuous differentiable respect to $\delta$. That is, the solution of $\mathcal{M}_{\delta}(\eta)=0$ can be written by
		\begin{equation}
			\eta_k = \eta_k^{(0)}+\eta_k^{(1)}\delta +\mathcal{O}(\delta^2), \quad k=1,\dots,N
		\end{equation}
		then
		\begin{equation}\label{eta1}
			\eta_1=(-1)^{N-1}(1-N(N-1)\delta)+\mathcal{O}(\delta^2),
		\end{equation}
		\begin{equation}
			\eta_k=(-1)^{N-k}+\mathcal{O}(\delta^2)\quad k=2,\dots,N.
		\end{equation}
		\begin{proof}
			With the help of Vandermonde determinant, straight forward computation shows
			\begin{equation}
				\begin{aligned}
					\xi_{N+1}&=(-1)^{N-1}\frac{\prod_{i\neq N}((1+\delta)^{2i}-1)}{\prod_{i\neq N}((1+\delta)^{2N}-(1+\delta)^{2i})}\\
					&=(-1)^{N-1}\frac{\prod_{i\neq N}(2i\delta+i(2i-1)\delta^2)}{\prod_{i\neq N}((2N-2i)\delta+(N(2N-1)-i(2i-1))\delta^2)}\\
					&=(-1)^{N-1}+(-1)^{N-1}\Big(\sum_{i=1}^{N-1}\big(\frac{2i-1}{2}\big)-\frac{2(N+i)-1}{2}\Big)\delta\\
					&=\xi_{N+1}^{(0)}+\xi_{N+1}^{(1)}\delta,
				\end{aligned}\notag
		\end{equation}
		where $\xi_{N+1}^{(0)}=(-1)^{N-1}$ and $\xi_{N+1}^{(1)}=(-1)^NN(N-1)$. Recalling the definition of $\xi_{k}$, we can immediately get \eqref{eta1}.
		Similar calculations yield that
		\begin{equation}
			\begin{aligned}
				\xi_{N}&=(-1)^{N-1}\frac{\prod_{i\neq N-1}((1+\delta)^{2i}-1)}{\prod_{i\neq N-1}((1+\delta)^{2N-2}-(1+\delta)^{2i})}\\
			&=\xi_{N}^{(0)}+\xi_{N}^{(1)}\delta,
			\end{aligned}\notag
		\end{equation}
		where $\xi_{N}^{(0)}=(-1)^{N-2}N$ and $\xi_{N}^{(1)}=(-1)^{N-1}N(N-1)^2$. We now rewrite $\xi_N$ as
		\begin{equation}
			\xi_N = \eta_2+\sum_{i=2}^N\eta_1\eta_i\eta_{i+1},
		\end{equation}
	    note that $|\eta_k|\leq 1,k=1,\dots,N$, we have $\eta_2^{(0)}=(-1)^{N-2}$ and $\eta_1^{(0)}\eta_i^{(0)}\eta_{i+1}^{(0)}=(-1)^{N-2}$. By induction one can easily obtain $\eta_k^{(0)}=(-1)^{N-k}$, thus we get $(-1)^{N-k}\eta_k^{(1)}\leq 0$. Furthermore the first order term can be written by
	    \begin{equation}
	    	\begin{aligned}
	    		\xi_{N}^{(1)}&=-(N-1)\eta_{1}^{(1)}+\sum_{i=2}^N(-1)^i2 \eta_{i}^{(1)}\\&=(-1)^{N-1}N(N-1)^2+\sum_{i=2}^N(-1)^i2 \eta_{i}^{(1)},
	    	\end{aligned}\notag
	    \end{equation}
    	therefore $\eta_k^{(1)}=0, k=2,\dots,N$.
		\end{proof}
	\end{theorem}
	\begin{remark}
		We note that the extreme vanishing structures of two-layers can be also considered as the high conductivity (HC) imperfect interface which studied in \cite{kang2019construction}. Indeed, it has been shown that if
		\begin{equation}\label{imperfect}
			\eta_1=1-\frac{2\sigma_2\sigma_0}{\sigma_2-\sigma_0}\delta+\mathcal{O}(\delta^2),
		\end{equation}
		then $\Omega$ is a 1-GPTs vanishing structure. For an insulative core with $\sigma_{2}=0$, the above condition \eqref{imperfect} simplifies to
		\begin{equation}
			\eta_1=1+\mathcal{O}(\delta^2). \notag
		\end{equation}
		which is consistent with \eqref{eta1}.
	\end{remark}
	\section{Numerical illustration}
	\label{sec:5}
	\quad \ \ In this section, we present some numerical examples to corroborate our theoretical results in the previous sections. We shall focus on two dimensional case, while for three dimensional case only entries of matrix should be changed accordingly. In order to find conductivity combination of $N$-GPTs vanishing structure, it suffices to solve the nonlinear equations by iteration methods.
	\begin{equation}\label{equations}
		\mathcal{M}_{N}(\eta)=0.
	\end{equation}

	For computation simplicity, we here use the formula generated by \cite{FD23,kong2024inverse}, that is
	\begin{equation}
		M_n[\eta]=-2\pi ne^T\Upsilon_N^{(n)}(P_N^{(n)}[\eta])^{-1}e,
	\end{equation}
	where $e=(1,1,\cdots,1)^T$ and $P_{N+1}^{(n)}[\eta]$, $\Upsilon_{N+1}^{(n)}$ are defined by \eqref{PN}, \eqref{RN}.
	We note that $P_{N+1}^{(n)}[\eta]$ is invertible for $\|\eta\|_{l^{\infty}}\leq 1$ due to the well-posedness of elliptic equation. Hence, simple calculation shows
	\begin{equation}
		\frac{\partial{M_n}}{\partial\eta_i}=-\frac{2\pi n}{\eta^2_i}e^T \Upsilon_{N+1}^{(n)}(P_{N+1}^{(n)}[\eta])^{-1}E_{ii}(P_{N+1}^{(n)}[\eta])^{-1}e,
	\end{equation}
	where $E_{ii}$ is sparse matrix with value 1 at the $i$-th row $i$-th column and $0$ others. As an illustration, we put $\eta_{N+1}=-1$, i.e., the core is insulated. To solve \eqref{equations} with the condition $\|\eta\|_{l^{\infty}}\leq 1$, we shall use the projection Newton iteration, the well known Newton direction, is described by
	\begin{equation}
		p_{N}=-(\nabla \mathcal{M}(\eta))^\dagger\mathcal{M}^T(\eta),
	\end{equation}
	here $(\mathcal{M}(\eta))^\dagger$ is the pseudo inverse of $\mathcal{M}(\eta)$. Note the projection Newton iteration is not always a descent iteration, we use the projection gradient
	iteration to ensure the convergence, and the gradient direction given by
	\begin{equation}
		p_{G}=-(\nabla \mathcal{M}(\eta))^T\mathcal{M}^T(\eta).
	\end{equation}
	In each iteration, we first identify a descent point along the Newton direction and then project it into the set
	$\|\eta\|_{l^{\infty}}\leq 1$. Next, we check whether the projected point is a descent point. If it is not, we repeat the process by replacing the gradient direction to ensure that the iteration remains target descent. Due to the local second-order convergence of Newton method, our iterative approach will get an accurate solution quickly.
	
	First, we consider that intervals between each two adjacent layers are equidistance. Set
	\begin{equation}\label{radii1}
		r_i=2-(i-1)/N,\quad i=1,2,\dots,N+1,
	\end{equation}
	such that $r_1=2$ and $r_{N+1}=1$. In Figure \ref{fig:GPTva1}, we show the conductivity distribution with the GPTs vanishing structure designed by \eqref{radii1}, where $N=8$ and $\Gs_{N+1}=0$. It can be seen the conductivity of adjacent layers behave an oscillatory pattern, mathematically, it's saying that $\sigma_{j}< \sigma_{j-1},\sigma_{j+1}$ for odd $j$ and $\sigma_{j}> \sigma_{j-1},\sigma_{j+1}$ for even $j$. As the layer approaches to core or background, the oscillation phenomenon will become more severe or slight. Moreover, the conductivity distribution also exhibits monotonicity in the gap layers, it shows the conductivity increase from slow to fast as approaching to core in even layers and decrease in odd layers.
	\begin{figure}
		\includegraphics[width=1\textwidth]{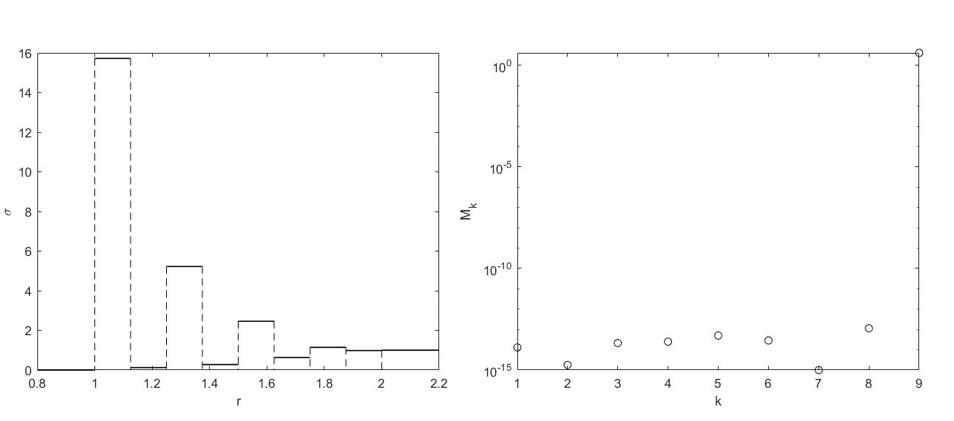}
		\captionsetup{font={small}}
		\caption{\label{fig:GPTva1} Left figure shows the conductivity distribution of GPTs vanishing structure with \eqref{radii1}, and right figure shows the values of CGPTs generated by such distribution. $N=8$ and $\Gs_{N+1}=0$.}
	\end{figure}

	Next, we consider the radius of layers are increasing with the same scale $\gamma$ which is studied in section \ref{sec:proportional}, that is
	\begin{equation}\label{radii2}
		r_{i-1} = \gamma r_{i},\quad i=1,2,\dots,N+1,
	\end{equation}
	for the convenience of comparison, we also set $r_1=2$ and $r_{N+1}=1$. Proportional radius settings makes the inner layers thinner, relatively, the outer layer will be thicker. Figure \ref{fig:GPTva2} presents the conductivity distribution with the $N$-GPTs vanishing structure designed by \eqref{radii2}, where $N=8$ and $\Gs_{N+1}=0$. Similarly, the oscillation phenomenon and monotonicity are also preserved in this configuration. Moreover, one can also see that the conductivity oscillation is more pronounced in thinner layers and slighter in thicker layers.
	\begin{figure}
		\includegraphics[width=1\textwidth]{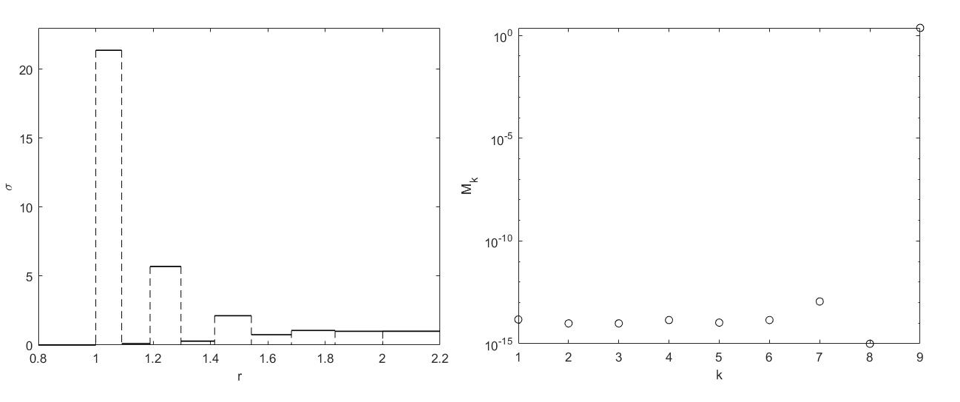}
		\captionsetup{font={small}}
		\caption{\label{fig:GPTva2}Left figure shows the conductivity distribution of GPTs vanishing structure with \eqref{radii2}, and right figure shows the values of CGPTs generated by such distribution.$N=8$, $\Gs_{N+1}=0$.}
	\end{figure}

	As the end, we verify our the conclusion of extreme case obtained in subsection \ref{subsec:extreme}. Take $\delta = 0.01$ and therefore $\gamma = 1.01$, Table \ref{table:example 1} exhibits the conductivity value $\sigma_k$ and contrast parameter $\eta_k$ with $N$-GPTs vanishing structure in extreme case, where $N=8$. The numerical result in the table is corresponding to our asymptotic conclusion. Moreover, the coats designed to achieve cloaking are quite thin, therefore the conductivity will be sufficient large. Indeed, this is also physically justifiable.
	
	\begin{table}[ht]
		\renewcommand{\arraystretch}{1}
		\caption{The GPTs vanishing structure designed in the extreme case.}
		\label{table:example 1}
		\centering
		\tiny
		\resizebox{1\textwidth}{!}{
			\begin{tabular}{cccccc}
				\toprule
				\multicolumn{6}{c}{$N=8$} \\
				\midrule
				$\eta_k$ & \textasciitilde &  -0.4485    &  0.9215  & -0.9844  & 0.9944 \\
				$M_k$& \textasciitilde  &  -1.5632e-13    &  2.5935e-13  & 0  & 1.4388e-13 \\
				$\sigma_k$&1  &  0.3436    &   8.4117  &  0.0661  & 23.3907\\
				\midrule
				$\eta_k$&-0.9975 &   0.9989  & -0.9996  & 0.9999 &  -1 \\
				$M_k$& 1.2790e-13 &   2.5757e-13  &  7.4607e-14  & 1.0000e-15  & 1.6874e-4 \\
				$\sigma_k$&0.0288 &   52.6041 &   0.0112   & 188.4053 &  0\\
				\bottomrule
			\end{tabular}
		}
	\end{table}
	It's worth to mention that the CGPTs are highly sensitive in extreme case, this is because the $p_{22}^{(n)}$ defined by \eqref{p22} will be sufficiently small. This singular phenomenon will greatly affect the stability of the numerical iteration. To overcome, we divide the interval $\eta_{N+1}\in [-1,0]$ into $d$ equal parts, that is, $0=\eta_{N+1}^1>\eta_{N+1}^2> \dots>\eta_{N+1}^d=-1$. We solve the accurate solution with fixed $\eta_{N+1}^j$ and take it as the initial value for $\eta_{N+1}^{j+1}$. The preprocessing technique can be a practical method for solving high-order GPTs vanishing setting. 
	
	\section*{Acknowledgements}
	The work of Y. Deng was supported by NSFC-RGC Joint Research Grant No. 12161160314.
	\section*{Data Availability Statement}
	No data was used for the research described in the article.

\section*{Conflict of interest statement}
The authors declared that they have no conflicts of interest to this work.
We declare that we do not have any commercial or associative interest that represents a conflict of interest in connection with the work submitted.
	

\small
	\begin{thebibliography}{00}
		\bibitem{ammari2013spectral} {H. Ammari, G. Ciraolo, H. Kang, H. Lee, and G. W. Milton}, Spectral theory of a
		Neumann–Poincaré-type operator and analysis of cloaking due to anomalous localized resonance, Arch. Rat. Mech. Anal., 208 (2013), 667-692.
		\bibitem{ammari2014reconstruction2} {H. Ammari, Y. Deng, H. Kang, and H. Lee}, Reconstruction of inhomogeneous conductivities
		via the concept of generalized polarization tensors, A. I. H. Poincar\'e-AN., 31
		 (2014), 877–897.
		\bibitem{ammari2013mathematical} {H. Ammari, J. Garnier, W. Jing, H. Kang, M. Lim, K. Sølna, and H. Wang}, Mathematical and statistical methods for multistatic imaging, 2098, Springer, 2013.
		\bibitem{ammari2007polarization} {H. Ammari and H. Kang}, Polarization and moment tensors: with applications to inverse
		problems and effective medium theory, 162, Springer Science, 2007.
		\bibitem{ammari2005reconstruction} {H. Ammari, H. Kang, E. Kim, and M. Lim}, Reconstruction of closely spaced small inclusions,
		SIAM J. Numer. Anal., 42 (2005), 2408–2428.
		\bibitem{ammari2013enhancement1} {H. Ammari, H. Kang, H. Lee, and M. Lim}, Enhancement of near cloaking using generalized
		polarization tensors vanishing structures. Part I: The conductivity problem, Commun. Math.
		Phys., 317 (2013), 253–266.
		\bibitem{ammari2013enhancement2} {H. Ammari, H. Kang, H. Lee, and M. Lim}, Enhancement of near-cloaking. Part II: The
		helmholtz equation, Commun. Math. Phys., 317 (2013), 485–502.
		\bibitem{astala} {K. Astala, M. Lassas, and L. Päivärinta}, Calderon’s inverse problem for anisotropic
		conductivity in the plane, Comm. Part. Diff. Equa., 30 (2004), 207–224.
		\bibitem{astala2006calderon} {K. Astala and L. Päivärinta}, Calderón’s inverse conductivity problem in the plane, Ann.
		Math., (2006), 265–299.
		\bibitem{bao2014nearly} {G. Bao, H. Liu, and J. Zou}, Nearly cloaking the full maxwell equations: cloaking active
		contents with general conducting layers, J. Math. Pures. Appl., 101 (2014), 716–733.
		\bibitem{benveniste1999neutral} {Y. Benveniste and T. Miloh}, Neutral inhomogeneities in conduction phenomena, J. Mech.
			Phys. Solids, 47 (1999), 1873–1892.
		\bibitem{blaasten2020recovering} {E. Blåsten and H. Liu}, Recovering piecewise constant refractive indices by a single far-field
		pattern, Inverse Problems, 36 (2020), 085005.
		\bibitem{bruhl2003direct} {M. Brühl, M. Hanke, and M. S. Vogelius}, A direct impedance tomography algorithm for
		locating small inhomogeneities, Numer. Math., 93 (2003), 635–654.
		\bibitem{choi2023geometric} {D. Choi, J. Kim, and M. Lim}, Geometric multipole expansion and its application to semineutral inclusions of general shape, Z. Angew. Math. Phys., 74 (2023), 39.
		\bibitem{choi2024construction} {D. Choi and M. Lim}, Construction of inclusions with vanishing generalized polarization tensors
		by imperfect interfaces, Stud. Appl. Math., 152 (2024), 673–695.
		\bibitem{ciarlet2013linear} {P. G. Ciarlet}, Linear and nonlinear functional analysis with applications, SIAM, (2013).
\bibitem{DLU171}
Y. Deng, H. Liu and G. Uhlmann, On regularized full- and partial-cloaks in acoustic scattering, Commun. Part. Differential Eq., 42 (6) (2017), 821-851.
\bibitem{DLU172}
Y. Deng, H. Liu and G. Uhlmann, Full and partial cloaking in electromagnetic scattering,  Arch. Ration. Mech. Anal., 223 (1)(2017), 265-299.

		\bibitem{deng2024identifying} {Y. Deng, H. Liu, and Y. Wang}, Identifying Active Anomalies in a Multilayered Medium by
		Passive Measurement in EIT, SIAM J. Appl. Math., 84 (2024), 1362–1384.
\bibitem{FD23}
 X. Fang and Y. Deng, On plasmon modes in multi-layer structures, Math. Meth. Appl. Sci. 46 (17) (2023), 18075-18095.
		\bibitem{feng2017construction} {T. Feng, H. Kang, and H. Lee}, Construction of gpt-vanishing structures using shape derivative, J. Comput. Phys., (2017), 569–585.
		\bibitem{greenleaf} {A. Greenleaf, M. Lassas, and G. Uhlmann}, On nonuniqueness for calderon’s inverse
		problem, Math. Res. Lett., 10 (2003), 685–693.
		\bibitem{jarczyk2012neutral} {P. Jarczyk and V. Mityushev}, Neutral coated inclusions of finite conductivity, Proc. Math.
		Phys. Eng. Sci., 468 (2012), 954–970.
		\bibitem{ji2021neutral} {Y.-G. Ji, H. Kang, X. Li, and S. Sakaguchi}, Neutral inclusions, weakly neutral inclusions,
		and an over-determined problem for confocal ellipsoids, in Geometric Properties for Parabolic
		and Elliptic PDE’s, Springer, 2021, 151–181.
		\bibitem{kang2019construction} {H. Kang and X. Li}, Construction of weakly neutral inclusions of general shape by imperfect
		interfaces, SIAM J. Appl. Math., 79 (2019), 396–414.
		\bibitem{kang2021polarization} {H. Kang, X. Li, and S. Sakaguchi}, Polarization tensor vanishing structure of general shape:
		Existence for small perturbations of balls, Asymptot. Anal., 125 (2021), 101–132.
		\bibitem{kang2022existence} {H. Kang, X. Li, and S. Sakaguchi}, Existence of weakly neutral coated inclusions of general
		shape in two dimensions, Appl. Anal., 101 (2022), 1330–1353.
		\bibitem{kohn2010cloaking} {R. V. Kohn, D. Onofrei, M. S. Vogelius, and M. I. Weinstein}, Cloaking via change of
		variables for the helmholtz equation, Commun. Pure. Appl. Math., 63 (2010), 973–1016.
		\bibitem{kohn2008cloaking} {R. V. Kohn, H. Shen, M. S. Vogelius, and M. I. Weinstein}, Cloaking via change of
		variables in electric impedance tomography, Inverse Problems, 24 (2008), 015016.
		\bibitem{kong2024inverse} {L. Kong, Y. Deng, and L. Zhu}, Inverse conductivity problem with one measurement: uniqueness of multi-layer structures, Inverse Problems, 40 (2024), 085005.
		\bibitem{kong2024enlargement} {L. Kong, L. Zhu, Y. Deng, and X. Fang}, Enlargement of the localized resonant band gap
		by using multi-layer structures, J. Comput. Phys., 518 (2024), 113308.
		\bibitem{leonhardt2009broadband} {U. Leonhardt and T. Tyc}, Broadband invisibility by non-euclidean cloaking, Science, 323
		(2009), 110–112.
		\bibitem{li2016literature} {J. Li}, A literature survey of mathematical study of metamaterials, Int. J. Numer. Anal. Model,
		13 (2016), 230–243.
\bibitem{LLRU15}
Li, J., Liu, H., Rondi, L., Uhlmann, G., Regularized transformation-optics cloaking
for the Helmholtz equation: from partial cloak to full cloak. Commun. Math. Phys. 335,
671–712 (2015)
		\bibitem{lim2020inclusions} {M. Lim and G. W. Milton}, Inclusions of general shapes having constant field inside the core
		and nonelliptical neutral coated inclusions with anisotropic conductivity, SIAM J. Appl. Math.,
		80 (2020), 1420–1440.
		\bibitem{nachman1996global} {A. I. Nachman}, Global uniqueness for a two-dimensional inverse boundary value problem,
		Ann. Math., (1996), 71–96.
		\bibitem{nguyen2016cloaking} {H.-M. Nguyen}, Cloaking using complementary media in the quasistatic regime, ANN. I. H.
		POINCARE-AN., 33 (2016), 1509–1518.
		\bibitem{pham2018solutions} {D. C. Pham}, Solutions for the conductivity of multi-coated spheres and spherically symmetric
		inclusion problems, Z. Angew. Math. Phys., 69 (2018), 1–14.
		\bibitem{schiffer1949virtual} {M. Schiffer and G. Szegö}, Virtual mass and polarization, Trans. Am. Math. Soc., 67 (1949),
		 130–205.
		\bibitem{wang2013neutral} {X. Wang and P. Schiavone}, A neutral multi-coated sphere under non-uniform electric field
		in conductivity, Z. Angew. Math. Phys., 64 (2013), 895–903.
		\bibitem{you2020combined} {L. You, J. Man, K. Yan, D. Wang, and H. Li}, Combined fourier-wavelet transforms for
		studying dynamic response of anisotropic multi-layered flexible pavement with linear-gradual
		interlayers, Appl. Math. Model, 81 (2020), 559–581.
\bibitem{Ze86}
E. Zeidler, Nonlinear functional analysis and its applications: Fixed point theorems, Springer-
Verlag, New York, 1986.
	\end{thebibliography}
\end{document}